\documentclass[11pt]{amsart}
\usepackage{pstricks}
\usepackage{graphicx}

\newtheorem{thm}{Theorem}[section]

\newtheorem{defn}[thm]{Definition}
\newtheorem{lemma}[thm]{Lemma}

\newtheorem{cor}[thm]{Corollary}
\newtheorem{remark}[thm]{Remark}
\newtheorem{example}[thm]{Example}

\usepackage{amsmath}
\usepackage{amsxtra}
\usepackage{amscd}
\usepackage{amsthm}
\usepackage{amsfonts}
\usepackage{amssymb}
\usepackage{eucal}
\newcommand{\lo}{{\ell}}
\newcommand{\sh}{{s}}
\newcommand{\cQ}{\mathcal Q}
\newcommand{\cA}{\mathcal A}
\newcommand{\cG}{\mathcal G}
\newcommand{\cT}{\mathcal T}
\newcommand{\tcT}{\widetilde{\mathcal T}}

\newcommand{\g}{{\mathfrak{g}}}

\renewcommand{\sl}{{\mathfrak{sl}}}
\newcommand{\KR}{{{\rm KR}}}

\newcommand{\C}{{\mathbb C}}
\newcommand{\Z}{{\mathbb Z}}

\newcommand{\ba}{{\mathbf a}}
\newcommand{\bm}{{\mathbf m}}
\newcommand{\bn}{{\mathbf n}}
\newcommand{\bb}{{\mathbf b}}
\newcommand{\bx}{{\mathbf x}}

\newcommand{\bu}{{\mathbf u}}
\newcommand{\al}{{\alpha}}
\newcommand{\tQ}{{R}}
\newcommand{\sgn}{{\rm sgn}}

\numberwithin{equation}{section}

\begin{document}
\title[$Q$-systems as cluster algebras II]{$Q$-systems as cluster
  algebras II: Cartan matrix of finite type
  and the polynomial property}
\author{P. Di Francesco and R. Kedem}
\address{PDF: Institut de Physique Th\'eorique de Saclay, CEA/DSM/IPhT,
  URA 2306 du CNRS, F-91191 Gif sur Yvette Cedex France.
  E-mail:  philippe.di-francesco@cea.fr}
\address{RK: Department of Mathematics, University of Illinois, 1409 West
  Green Street, Urbana, IL 61821. e-mail: rinat@uiuc.edu}

\date{\today}
\begin{abstract}
  We define the cluster algebra associated with the $Q$-system for the
  Kirillov-Reshetikhin characters of the quantum affine algebra
  $U_q(\widehat{\g})$ for any simple Lie algebra $\g$, generalizing
  the simply-laced case treated in [Kedem 2007].  We describe some
  special properties of this cluster algebra, and explain its relation
  to the deformed $Q$-systems which appeared on our proof of the
  combinatorial-KR conjecture. We prove that the polynomiality of the
  cluster variables in terms of the ``initial cluster seeds'',
  including solutions of the $Q$-system, is a consequence of the
  Laurent phenomenon and the boundary conditions. We also give a
  formulation of both $Q$-systems and generalized $T$-systems as
  cluster algebras with coefficients. This provides a proof of the
  polynomiality of solutions of generalized $T$-systems with
  appropriate boundary conditions.
\end{abstract}
\maketitle
\section{Introduction}
The $Q$-system is a recursion relation for the characters of certain
finite-dimensional representations of the quantum affine algebra
$U_q(\widehat\g)$ or the Yangian $Y(\g)$, where $\g$ is a simple Lie
algebra. 

$Q$-systems were introduced in \cite{KR} for the classical algebras.
They were later generalized by \cite{HKOTY} for the exceptional
algebras and later to more complicated situations, such as twisted
quantum affine algebras \cite{HKOTY2} and double affine algebras
\cite{Her07}. 

The special modules related to $Q$-systems are called \cite{KunibaYsys,KN}
Kirillov-Reshetikhin modules. The fact that their characters
satisfy the $Q$-system was proved by Kirillov and Reshetikhin
\cite{KR} for $\g=A_r$, by Nakajima \cite{Nakajima} for the
simply-laced algebras and by Hernandez \cite{Hernandez,Her07} in further
generality.

Cluster algebras were introduced by Fomin and Zelevinsky
\cite{FZcluster} in 2000, and are a very general algebraic tool which
has since been applied in various algebraic, combinatorial and
geometric contexts. In particular, they have been used to study
$Y$-systems \cite{FZysys}, which are related to $Q$-systems in the
sense that both can be derived from $T$-systems \cite{KunibaYsys}. The
$T$-systems are a consequence of the fusion relations for Yangian or
quantum algebra modules.

The form of the $Q$-system suggests that it should be possible to
recast it as part of a cluster algebra. The first step in this
reformulation, for the case where $\g$ is simply-laced, was derived in
\cite{Ke07}. In the current
article, we give a generalization of this case to non-simply laced
$\g$.  In this formulation, the $Q$-system appears in a very simple
and easily generalizable form. We note that it does not appear to be
directly related to the cluster algebra coming from the $Y$-system
studied by Fomin and Zelevinsky \cite{FZysys}.

In \cite{DFK}, we proved a combinatorial identity (``the $M=N$''
conjecture of \cite{HKOTY}) which implies the proof of the
combinatorial Kirillov-Reshetikhin conjecture. In our proof, we
introduced what we called the deformed $Q$-system, depending on an
increasing number of formal variables. A specialization of this system
can be expressed as the $Q$-system with general boundary conditions
(not corresponding to Kirillov-Reshetikhin characters), or as a
cluster algebra. It is essential for the formulation as a cluster
algebra to have this more general system.

Our proof of the $M=N$ identity depends crucially on the fact that
the KR-characters are polynomials in the fundamental KR-characters.
This can be rephrased in terms of the cluster algebra, by considering
a specialization of the initial cluster variables to the special point
which gives the KR-characters at the other nodes of the cluster graph.
We call this specialization of the initial parameters the KR point.

On the other hand, it is known that cluster variables obey the Laurent
phenomenon \cite{FZlaurent}.  We show that, under the specialization
of the cluster variables to the KR point, this becomes what we called
in \cite{Ke07} the strong Laurent phenomenon. That is, the cluster
variables on the entire cluster graph, not just the subgraph
corresponding to the $Q$-system, are polynomials in the initial
cluster variables.

The paper is organized as follows. We recall the definition of
$Q$-systems for any simple Lie algebra in Section 2, as well as the
definition of normalized cluster algebras without coefficients.
Section 3 deals with the formulation of each $Q$-system as a subgraph
in a cluster algebra. In Section 3.1, we review the results of
\cite{Ke07} about the formulation of $Q$-systems as cluster algebras
for simply-laced Lie algebras.  In Sections 3.2 and 3.3, we formulate
the cluster algebras corresponding to the non-simply laced simple Lie
algebras. In Section 4, we prove that the special boundary conditions
which give solutions of the $Q$-system as characters of
Kirillov-Reshetikhin modules imply the polynomiality of the cluster
variables as functions of the seed variables at the boundary node.
Section 5 is a discussion of the results. Appendix A is a
reformulation of the $Q$-system as a cluster algebra with
coefficients, addressing the technical point of subtraction-free
expressions. In the body of the paper, this is done through a
renormalization of the cluster variables, which is not always
generalizable.
Appendix B is a discussion of the formulation of generalized
$T$-systems as cluster algebras, with two main examples, both of which
have the polynomiality property.

\vskip.2in
\noindent{\bf Acknowledgements.} We thank Bernhard
Keller, Hugh 
Thomas and especially Sergei Fomin for their valuable input. RK thanks 
CEA-Saclay IPhT for their hospitality. We thank the organizers of the MSRI
program on ``Combinatorial Representation Theory'' for their
hospitality. RK is supported by NSF grant
DMS-05-00759. PDF acknowledges the support of European Marie Curie
Research Training Networks ENIGMA MRT-CT-2004-5652,  ENRAGE
MRTN-CT-2004-005616, ESF program MISGAM, and of ANR
program GIMP ANR-05-BLAN-0029-01.

\section{Definitions}

\subsection{The $Q$-system}
\subsubsection{$\KR$-modules}
Let $\g$ be a simple Lie algebra of rank $r$ and let
$I_r=\{1,...,r\}$ be the parameterizing set for the simple roots of
$\g$. Let $C$ be the Cartan matrix of $\g$.

For any such algebra, there is a corresponding quantum affine algebra
$U_q(\widehat{\g})$ and a Yangian algebra $Y(\g)$. In the study of the
finite-dimensional modules of these algebras, there are certain
special modules, called Kirillov-Reshetikhin (KR) modules \cite{KR}.
These can be defined in terms of their Drinfeld polynomials
\cite{ChariPres}, for example, which have a particularly simple form
for these modules.

The modules are parameterized by $\g$-highest weights of the form $k
\omega_\al$, that is, a multiple of one of the fundamental weights, as
well as a complex parameter $\zeta$. We refer to such a module by
$W_{\al,k}(\zeta)$ where $\al\in I_r$, $k\in \Z_+$ and $\zeta\in
\C^*$.  These are finite-dimensional modules, which are analogs of
evaluation modules of affine algebras, which are obtained from
inducing from the irreducible representation of the finite-dimensional
subalgebra $\g$, which does not depend on $\zeta$. Unlike such
evaluation modules, however, they are not necessarily irreducible when
restricted to the subalgebra $U_q(\g)\subset U_q(\widehat{\g})$ or
$\g\in Y(\g)$, except in special cases, such as when $\g\neq
A_r$. The idea is that $W_{\al,k}(\zeta)$ is the ``smallest''
finite-dimensional $\g[t]$-module, with largest component $V(k\omega_\al)$ (the
irreducible $\g$-module), which deforms to a Yangian module.

KR-modules have been extensively studied in recent years, and
much is known about their properties. In particular, their
decomposition, and the decomposition of their tensor products,
into irreducible $U_q(\g)$-modules is known in terms of explicit
multiplicity formulas.
These formulas can be derived \cite{KR,HKOTY,Hernandez,DFK}
from the $Q$-system, a recursion relation for the characters
$$Q_{\al,k}={\rm char}~ {\rm Res}^{U_q(\widehat{\g})}_{U_q(\g)}
W_{\al,k}(\zeta),$$
in addition to a certain asymptotic property (see (C) of Theorem 8.1
of \cite{HKOTY}). We leave out the parameter $\zeta$ in what follows,
because it only affects the action of the affine part of the algebra
and not $U_q(\g)$ or the characters $Q_{\al,k}$.

The formulas for the decomposition coefficients and the statement that
the characters of $KR$-modules satisfy the $Q$-system were both known
(until they were proven) as the Kirillov-Reshetikhin conjecture.

\subsubsection{$Q$-systems}
We introduce the $Q$-system in two steps. First, let us write down
a completely general relation for a family of variables 
$$\{Q_{\al,k}| ~ \al\in I_r, k\in \Z\}.$$
Let $\g$ be a simple Lie algebra with Cartan matrix $C$. We denote the
simple roots $\al$ by the corresponding integers in $I_r=\{1,...,r\}$. The
$Q$-system associated with $\g$ is a recursion relation of the form
\begin{equation}\label{Qsystem}
  Q_{\al,k+1} Q_{\al,k-1} = Q_{\al,k}^2 - \prod_{\beta\sim\al}
  \cT^{(\al,\beta)}_k, ~\al\in I_r, ~k\in \Z,
\end{equation}
where $\al\sim\beta$ means that $\al$ is connected to $\beta$ in the
Dynkin diagram, and
\begin{equation}\label{T}
\cT^{(\al,\beta)}_k = \prod_{i=0}^{|C_{\al,\beta}|-1}
Q_{\beta,\lfloor\frac{t_\beta k+i}{t_\al} \rfloor},
\end{equation}
where $\lfloor a \rfloor$ is the integer part of $a$.
Here, $t_\al$ are the integers which symmetrize the Cartan
matrix. Namely, $t_r=2$ for $B_r$, $t_\al=2$ ($\al<r$) for $C_r$,
$t_3=t_4=2$ for $F_4$ and $t_2=3$ for $G_2$, and $t_\al=1$ otherwise.

This is a recursion relation which can be used to express
any $Q_{\al,k}$ in terms of $2r$ initial data points. Usually, one
chooses the initial conditions $\{Q_{\al,0},
Q_{\al,1}~|~\al\in I_r\}$. 
Then
\begin{equation}\label{positive}
Q_{\al,k+1} = \frac{Q_{\al,k}^2 - \prod_{\beta\sim \al}
  \cT_k^{\al,\beta}}{Q_{\al,k-1}}, \ k\geq 1,
\end{equation}
and
\begin{equation}\label{negative}
Q_{\al,k-1}=\frac{Q_{\al,k}^2 - \prod_{\beta\sim \al}
  \cT_k^{\al,\beta}}{Q_{\al,k+1}}, \ k\leq0.
\end{equation}

Consider Equation \eqref{positive} in the particular case of the initial
conditions that for all $\al$, $Q_{\al,0}=1$ and $Q_{\al,1}={\rm Char}
W_{\al,1}$, the character of 
the fundamental Kirillov-Reshetikhin module, with
highest $\g$-weight $\omega_\al$. In this case, the
solutions $Q_{\al,k}$ of \eqref{positive} are known to be the
characters of KR-modules $W_{\al,k}$ with highest $\g$-weight $k
\omega_\al$ \cite{Nakajima,Hernandez}. 
The recursion relation \eqref{positive} with these initial conditions
is what is usually called the $Q$-system in the literature.

We do not impose these initial conditions in the definition of
the cluster algebra. As we will see, this is a special (singular)
point in the parameter space for this algebra. One way to see that it
is singular is that equation \eqref{negative} cannot be used, in this
case, to define all the variables $Q_{\al,k}$ with $k<0$ because, for
example, $Q_{\al,-1}=0$ in this case.

More general initial conditions, of the form
\begin{equation}\label{deformedinitial}
Q_{\al,0}=b_\al, \ Q_{\al,1}=a_\al
\end{equation}
for some formal variables $a_\al, b_\al$ give a solution of (a special
case of) the ``deformed Q-system'' introduced in \cite{DFK}. We will
explain this point in Section \ref{discussion}.

\subsubsection{Normalized $Q$-systems}
In this paper we deal with cluster algebras without coefficients. To
see that it is possible to reformulate the $Q$-system in this way, 
a version of the following Lemma was given in \cite{Ke07}:
\begin{lemma}\label{renormalize}
There exist complex numbers $\epsilon_1,...,\epsilon_r$ (all of them roots of
1) such that the variables $\tQ_{\al,i}=\epsilon_\al Q_{\al,i}$
satisfy the normalized $Q$-system
\begin{equation}\label{tQsystem}
\tQ_{\al,k+1} \tQ_{\al,k-1} = \tQ_{\al,k}^2 + \prod_{\beta\sim\al}
\widetilde{\cT}^{(\al,\beta)}_k, 
\end{equation}
where
$$
\widetilde{\cT}^{(\al,\beta)}_k = \prod_{i=0}^{|C_{\al,\beta}|-1}
\tQ_{\beta,\lfloor\frac{t_\beta(k+i)}{t_\al} \rfloor}.
$$
\end{lemma}
\begin{proof}
  From Equation \eqref{T} we note that the total degree of
  $Q_{\beta,j'}$ (for various $j'$) in $\cT^{(\al,\beta)}_k$ is always
  $|C_{\al,\beta}|$, where $C$ is the Cartan matrix of $\g$. Therefore
  $\tQ_{\al,k}=\epsilon_\al Q_{\al,k}$
  satisfy the following system:
$$
\tQ_{\al,k+1}\tQ_{\al,k-1}=\tQ_{\al,k}^2 - (\prod_{\beta}
\epsilon_\beta^{C_{\al,\beta}})\prod_{\beta\sim\al}
\widetilde{\cT}^{(\al,\beta)}_k,
$$
where ${\tcT}$ is just $\cT$ with the variables $Q$ replaced
by $\tQ$.

Thus, if the variables $\epsilon_\al$ satisfy
$$
\prod_{\beta=1}^r \epsilon_\beta^{C_{\al,\beta}} = -1
$$
then \eqref{tQsystem} is satisfied.
But since $C$ is a Cartan matrix of finite type, it is invertible, and
this system has a solution. Namely if we set $\epsilon_\beta = e^{i
  \pi \mu_\beta}$, then $\mu_\al = \sum_\beta C^{-1}_{\beta,\al}$.
\end{proof}

In fact, we find $\epsilon_\beta^4=1$ in general. We note that this
choice of normalization is done in order to make contact with cluster
algebras without coefficients below. We do not believe that it is
essential for the final result.

We will refer to \eqref{tQsystem} as the normalized $Q$-system or
simply the $Q$-system when it is clear that we are considering the
evolution of the variables $\tQ$.

\subsection{Cluster algebras}

We use here the definition of \cite{FZfour} to define a cluster
algebra without coefficients.

\begin{defn}{A cluster algebra.}
Fix $n$ and define a labeled $n$-ary tree $\mathbb T_n$, with nodes
labeled by some parameter $t$ and $n$ edges labeled $1,...,n$
emanating from each node.

To each node $t$ we associate a seed $(\bx, B)_t$, where
$\bx = (x_1,...,x_n)$ (the cluster variables) and $B$ (the
mutation matrix) is a skew-symmetric $n\times n$
matrix with integer coefficients.

If the node $t$ is connected to the node $t'$ by an edge labeled by
$k$, then $(\bx[t'],B[t'])=\mu_k(\bx[t],B[t])$ where
we have the following relation (mutations or evolution equations along
the edge $k$) between the cluster at $t$ and $t'$:
\begin{eqnarray}
\mu_k(x_i) &=& \left\{ \begin{array}{ll} x_i, & i\neq k; \\
x_i^{-1}(\prod_j x_j^{[B_{jk}]_+} + \prod_{j} x_j^{[-B_{jk}]_+}). &
i=k.\end{array}\right.\\
\mu_k(B_{ij}) &=& \left\{ \begin{array}{ll}
-B_{ij}, & \hbox{if $i=k$ or $j=k$}; \\
B_{ij} + {\rm sgn}(B_{ik})[B_{ik}B_{kj}]_+, &
\hbox{otherwise}.\end{array}\label{Bmutations}\right. 
\end{eqnarray}
where $[n]_+$ denotes the positive part of $n$ and 
$$
\sgn(n)=\left\{ \begin{array}{ll} 0, & n=0;\\
1, & n>0; \\
-1, & n<0.\end{array}\right.
$$
\end{defn}

One can check the following properties of cluster mutations:
\begin{itemize}
\item The mutations \eqref{Bmutations} preserve the skew-symmetry of
$B$. 
\item If $B_{ij}=0$ then the mutations $\mu_i$ and $\mu_j$
commute. 
\item If $B_{ij}=0$ and
$k,l$ are two integers distinct from both $i$ and $j$, then
\begin{equation}\label{commutingcompound}
\mu_i\circ \mu_j( B_{kl}) = B_{kl} + \sgn(B_{ki})[B_{ki}B_{il}]_+
+ \sgn(B_{kj})[B_{kj}B_{jl}]_+.
\end{equation}
This statement can be extended to any finite sequence of commuting
mutations.
\item Mutations are defined in such a way that
$\mu_i^2=1$. Thus, cluster graphs are not oriented.
\end{itemize}

\begin{defn}
We define $\cA$ to be the family of variables
$\{\tQ_{\al,i}\}_{i\in \Z,\al\in I_r}$ which satisfy the normalized
$Q$-system \eqref{tQsystem}.
\end{defn}

In what follows we are interested in subgraphs of
$\mathbb T_n$ which encode the $Q$-system, that is, the cluster
variables at the nodes are elements of $\cA$ and the mutations between
nodes are $Q$-system evolutions.

More precisely, we will define the quotient graphs $\cG_\g$, by
identifying nodes corresponding to the same cluster variables and
mutation matrices. These are infinite graphs, but they have certain
periodicity properties and are therefore easy to describe
explicitly.

Our starting goal is to find a graph $\cG_\g$ for each algebra,
and the corresponding cluster variables, which encode the full
$Q$-system, in the following sense: The union over all nodes of
$\cG_\g$ of the cluster variables is the full set $\cA$.

Of course once the seed $(\bx,B)$ is given at any node $t$, then the
cluster algebra is defined, that is, the cluster at any other node is
defined via the evolution equations. Therefore, our definition gives a
cluster algebra on the full tree $\mathbb T_n$. We will have something
to say about special properties of the cluster variables on $\mathbb
T_n$ in Section \ref{poly}, but thus far they have no
representation-theoretical interpretation to our knowledge.

This definition of $\cG_\g$ allows us to encode the $Q$-system inside
a cluster algebra.  The graphs $\cG_\g$ which we describe
are usually subgraphs of a larger component of $\mathbb T_n$ with the
same property, and therefore our definition is not unique. There may
be other nodes of $\mathbb T_n$ with cluster variables in $\cA$ which
are not in $\cG_\g$.  This happens if the rank of $\g$ is greater than
two. We will give the example of $A_3$ to make this point clearer in
the discussion.

\section{The $Q$-systems as cluster algebras for simple
  $\g$}\label{qascluster} 

Let $C$ be a Cartan matrix of a simple Lie algebra $\g$ of rank $r$,
and let $B$ be the $2r \times 2r$
matrix defined in the block form:
\begin{equation}\label{Bmatrix}
B = \left(\begin{array}{cc} A & -C^t\\   C & 0 \end{array}\right),
\end{equation}
where $A=C^t-C$. This is a skew-symmetric matrix of integers.

For each $k\in \Z$, define
$\bx[k] = (x_1,...,x_{2r})$, where
\begin{equation}\label{clustervariable}
x_\al = \left\{\begin{array}{ll} \tQ_{\al, 2 t_\al k}, & \al\leq r \\
\tQ_{\al-r,2 t_\al k + 1}, & r<\al\leq 2r \end{array}\right.
\end{equation}
Here, $t_\al=2$ for the short roots of $B_r, ~C_r$ and $F_4$,
$t_\al=3$ for the short root of $G_2$, and and $t_\al=1$ otherwise.

In our cluster algebra graph, nodes labeled by $k\in \Z$ have associated
with them cluster $(\bx[k], B)$. We will prove the following
theorem: 

\begin{thm}
  There exists a cluster graph $\cG_\g$, which includes all nodes
  labeled by $k\in \Z_+$, with corresponding cluster variables $\bx[k]$
  as in \eqref{clustervariable} and $B$ as in \eqref{Bmatrix}, such
  that all mutations in the graph are $Q$-system evolutions. Moreover,
  the set $\cA$ is equal to the union of the cluster variables over
  all nodes in $\cG_\g$.
\end{thm}

We will describe such a graph for each algebra $\g$ below.

In all cases considered below, if the mutations
$\{\mu_{i_1},...,\mu_{i_n}\}$ commute starting from node $t$ (that is,
if $B_{i_k,i_j}=0$ for all $k,j$ at the node $t$), then we may act with these
mutations in any order, as long as they are distinct, and nodes
corresponding to acting with a certain set of mutations, regardless of
order, can be
identified with each other in the quotient graph $\cG_\g$ (this is
similar in spirit to the bipartite graph of \cite{FZfour}). These
mutations give rise to an n-simplex (e.g. a cube in the case $n=3$) in
$\cG_\g$.

We will in general describe only the cluster variables reached via
certain compound mutations
$\mu^{(i)}=\mu_{i_1}\circ\cdots\circ\mu_{i_n}$, where the
corresponding mutations commute. The cluster graph $\cG_\g$ is
recovered by replacing a compound mutation with the corresponding
$n$-simplex.

\subsection{The simply laced case}
This case was considered in \cite{Ke07}. We review the results here in
order to introduce the notations.

If $\g$ is simply-laced, the Cartan matrix is symmetric, $C^t=C$, and
therefore the matrix $A$ vanishes and $B$ is block-diagonal:
\begin{equation}\label{Bsimple}
B=\left(\begin{array}{rr} 0 & -C \\ C & 0 \end{array}\right).
\end{equation}

The graph
$\cG_\g$ has a particularly simple form in this case.
It includes the special nodes $k$ ($k\in \Z_+$) with cluster variables 
\begin{equation}\label{simpleclus}
\bx[k]
= (\tQ_{1,2k},...,\tQ_{r,2k},\tQ_{1,2k+1},...,\tQ_{r,2k+1}).
\end{equation} 
In addition, it includes the special nodes $k'$, which are reached
from $k$ by the compound mutation $\mu^{(1)} = \mu_1\circ \cdots \circ
\mu_r$. These commute due to the block-diagonal structure of $B$. The
following was proved in \cite{Ke07}:
\begin{thm}\label{simple}
For $k\in \Z_+$, Let $\bx[k]$ be as in \eqref{simpleclus} and $B$ be
the matrix in \eqref{Bsimple}. Then
\begin{enumerate}
\item At the node $k$, the mutations $\mu_1,...,\mu_r$ commute among
  themselves, and evolution according to these mutations corresponds
  to the $Q$-system evolutions $\tQ_{\al,2k}\mapsto \tQ_{\al,2k+2}$,
  $\al\in I_r$.
\item At the same node, the mutations $\mu_{r+1},...,\mu_{2r}$ also
  commute among themselves, and applying them to the variables
  $\bx[k]$ corresponds to the $Q$-system evolution
  $\tQ_{\al,2k+1}\mapsto \tQ_{\al,2k-1}$, $\al\in I_r$.
\item Define the compound mutation
\begin{equation*}
\mu^{(I_r)}=\mu_1\circ \cdots \circ \mu_r
\end{equation*}
and denote the node reached by this compound mutation acting on the
cluster variables at $k$ by $k'$. Then the mutation matrix
$B'=\mu^{(I_r)}(B)$ is
\begin{equation*}
B'=-B = \left(\begin{array}{rr} 0 & C \\ -C & 0\end{array}\right).
\end{equation*}
and 
$$\mu^{(I_r)}(\bx[k]) =
(\tQ_{1,2k+2},...,\tQ_{r,2k+2},\tQ_{1,2k+1},...,\tQ_{r,2k+1}) = \bx[k'].$$
\item At the point $k'$, the mutations $\mu_{r+1 },...,\mu_{2r}$
  commute among themselves, the evolution according to these mutations
  corresponds to the $Q$-system evolution $\tQ_{\al,2k+1}\mapsto
  \tQ_{\al,2k+3}$.
\item The compound mutation 
$$
\mu^{(I_r')} = \mu_{r+1}\circ\cdots\circ \mu_{2r}
$$
maps the seed at $k'$ to the seed at $k+1$ with cluster
variables $\bx[k+1]$ and mutation matrix $B$.
\end{enumerate}
\end{thm}
This theorem is proved by direct calculation. 

Let us illustrate some examples of graphs which correspond to
$\cG_\g$.

\begin{example}
The
graph $\cG_{\sl_2}$
is linear:
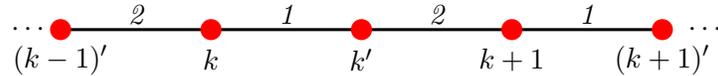
\begin{center}
\psset{unit=2mm,linewidth=.4mm,dimen=middle}
\begin{pspicture}(0,0)(50,10)
\psline(5,5)(45,5)
\multips(5,5)(10,0){5}{\pscircle*[linecolor=red](0,0){.7}}
\rput(5,3){$(k-1)'$}
\rput(15,3){$k$}
\rput(25,3){$k'$}
\rput(35,3){$k+1$}
\rput(45,3){$(k+1)'$}
\rput(3,5){$\cdots$}
\rput(48,5){$\cdots$}
\rput(10,6){2}
\rput(20,6){1}
\rput(30,6){2}
\rput(40,6){1}
\end{pspicture}
\end{center}
The cluster variables at the primed and unprimed nodes are described
by $\bx[k]=(\tQ_{1,2k},\tQ_{1,2k+1})$ and
$\bx[k']=(\tQ_{1,2k+2},\tQ_{1,2k+1})$ of Theorem \ref{simple}. The
numbers above each edge correspond to the mutations $\mu_1$ and
$\mu_2$. This is the full tree $\mathbb T_2$.
\end{example}

\begin{example}
The graph $\cG_{\sl_3}$ is slightly more interesting:
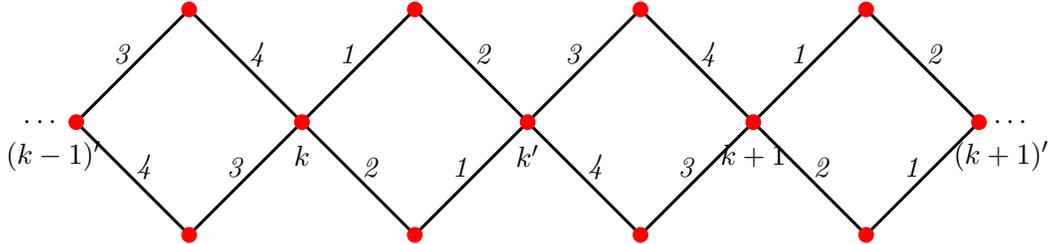
\begin{center}
\psset{unit=1.5mm,linewidth=.4mm,dimen=middle}
\begin{pspicture}(0,0)(80,30)
\multips(0,0)(20,0){4}{\psline(0,10)(10,20)(20,10)(10,0)(0,10)}
\multips(0,10)(20,0){5}{\pscircle*[linecolor=red](0,0){.7}}
\multips(10,0)(20,0){4}{\pscircle*[linecolor=red](0,0){.7}}
\multips(10,20)(20,0){4}{\pscircle*[linecolor=red](0,0){.7}}
\rput(-2,7){$(k-1)'$}
\rput(20,7){$k$}
\rput(40,7){$k'$}
\rput(60,7){$k+1$}
\rput(82,7){$(k+1)'$}
\rput(-3,10){$\cdots$}
\rput(83,10){$\cdots$}
\rput(6,6){4}
\rput(26,6){2}
\rput(46,6){4}
\rput(66,6){2}
\rput(4,16){3}
\rput(24,16){1}
\rput(44,16){3}
\rput(64,16){1}
\rput(16,16){4}
\rput(36,16){2}
\rput(56,16){4}
\rput(76,16){2}
\rput(14,6){3}
\rput(34,6){1}
\rput(54,6){3}
\rput(74,6){1}
\end{pspicture}
\end{center}
The cluster algebra has four variables, and four possible mutations,
which are the labels along the edges. The special nodes $k$ and their
primed versions appear on the center line of the graph. The other
nodes have intermediate cluster variables.
\end{example}

The full cluster algebra graph of $\sl_3$ would include, for example,
an edge labeled $4$ emanating from the node $\mu_1(k)$. However the cluster
mutation in this direction is not a $Q$-system evolution.

\begin{example}\label{Athree}
The graph $\cG_{\sl_4}$ described in Theorem \ref{simple} is as follows:
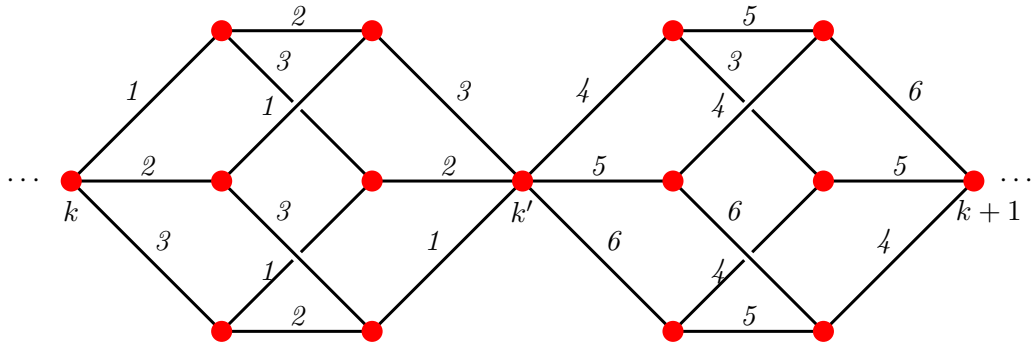
\begin{center}
\psset{unit=2mm,linewidth=.4mm,dimen=middle}

\begin{pspicture}(0,10)(60,33)
\psline(10,10)(0,20)(10,30)
\psline(10,10)(20,20)(10,30)
\psline[linewidth=1.5mm,linecolor=white](20,10)(10,20)(20,30)
\psline(20,10)(10,20)(20,30)
\psline(10,10)(20,10)
\psline(20,30)(10,30)
\psline(0,20)(10,20)
\psline(20,10)(30,20)(20,30)
\psline(40,30)(30,20)(40,10)
\psline(40,10)(50,20)(40,30)
\psline[linewidth=1.5mm,linecolor=white](50,10)(40,20)(50,30)
\psline(50,10)(40,20)(50,30)
\psline(20,20)(40,20)
\psline(50,10)(60,20)(50,30)
\psline(40,10)(50,10)
\psline(50,20)(60,20)
\psline(40,30)(50,30)
\multips(0,20)(10,0){7}{\pscircle*[linecolor=red](0,0){.7}}
\multips(10,10)(10,0){2}{\pscircle*[linecolor=red](0,0){.7}}
\multips(40,10)(10,0){2}{\pscircle*[linecolor=red](0,0){.7}}
\multips(10,30)(10,0){2}{\pscircle*[linecolor=red](0,0){.7}}
\multips(40,30)(10,0){2}{\pscircle*[linecolor=red](0,0){.7}}
\rput(30,0){\pscircle*[linecolor=white](0,0){.5}}
\rput(6,16){3}
\rput(5,21){2}
\rput(4,26){1}
\rput(15,11){2}
\rput(15,31){2}
\rput(13,14){1}
\rput(13,25){1}
\rput(14,18){3}
\rput(14,28){3}
\rput(26,26){3}
\rput(25,21){2}
\rput(24,16){1}
\rput(36,16){6}
\rput(35,21){5}
\rput(34,26){4}
\rput(45,11){5}
\rput(45,31){5}
\rput(43,14){4}
\rput(43,25){4}
\rput(44,18){6}
\rput(44,28){3}
\rput(56,26){6}
\rput(55,21){5}
\rput(54,16){4}
\rput(0,18){$k$}
\rput (30,18){$k'$}
\rput(61,18){$k+1$}
\rput(63,20){$\cdots$}
\rput(-3,20){$\cdots$}
\end{pspicture}
\end{center}

\vskip.1in
\end{example}

The graphs $\cG_\g$ which we describe are not, in
general, the complete subgraph of $\mathbb T_n$ corresponding to
$Q$-system evolutions, except when ${\rm rank}(\g)\leq 2$. We chose
the graphs $\cG_\g$ so that the union of all cluster variables at
their nodes corresponds to the complete set of normalized characters
of the $KR$-modules. These graphs are not unique, and their extensions
are of interest also (see section \ref{extensions}).

Clearly, the cluster graphs $\cG_\g$ are difficult to
draw as the rank becomes high. Therefore, in treating the higher-rank
algebras, it is useful to adopt a pictorial notation which is similar
in spirit to the bipartite graph notation of \cite{FZfour}.

First, we define the notion of a compound mutation to be a  product
of $n$ distinct commuting mutations. Let $I=\{i_1,...,i_n\}$ be the
index set of these mutations, and define $\mu^{(I)}=\mu_{i_1}\circ \cdots
\circ \mu_{i_n}$. The order of mutations is irrelevant by
definition. We may also use the notation $\mu^{(I)}=\prod_{i\in I}
\mu_i$. Obviously, since $\mu_i^2=1$, we preserve the property that
$(\mu^{(I)})^2=1$ for the product of commuting notations.

The edge:
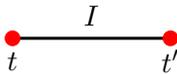
\begin{center}
\psset{unit=1.5mm,linewidth=.4mm,dimen=middle}
\begin{pspicture}(0,0)(20,7)
\psline(3,3)(17,3)
\multips(3,3)(14,0){2}{\pscircle*[linecolor=red](0,0){.7}}
\rput(10,5){$I$}
\rput(3,1){$t$}
\rput(17,1){$t'$}
\end{pspicture}
\end{center}
corresponds to the compound mutation $t' = \mu^{(I)}(t)$.  In the
graph $\cG_\g$, this corresponds to an $n$-simplex with node $t$ at
one end and node $t'$ at the other. It is a diamond in the case for
$\cG_{\sl_3}$ or a cube in the case of $\cG_{\sl_4}$.

Thus the graph $\cG_\g$ for
$\g$ simply-laced with the sequence of nodes and compound mutations as
follows: 
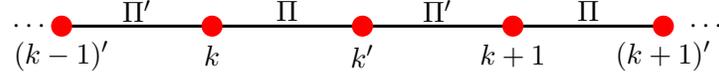
\begin{center}
\psset{unit=2mm,linewidth=.4mm,dimen=middle}
\begin{pspicture}(0,0)(50,10)
\psline(5,5)(45,5)
\multips(5,5)(10,0){5}{\pscircle*[linecolor=red](0,0){.7}}
\rput(5,3){$(k-1)'$}
\rput(15,3){$k$}
\rput(25,3){$k'$}
\rput(35,3){$k+1$}
\rput(45,3){$(k+1)'$}
\rput(3,5){$\cdots$}
\rput(48,5){$\cdots$}
\rput(10,6){$\Pi'$}
\rput(20,6){$\Pi$}
\rput(30,6){$\Pi'$}
\rput(40,6){$\Pi$}
\end{pspicture}
\end{center}
where $\Pi=I_r$ and $\Pi' = \{r+1,...,2r\}$.
The cluster
variables and mutation matrices at the nodes $k$, $k'$ are given by
Theorem \ref{simple}.

\subsection{The $Q$-system for the algebras $B_r,C_r,F_4$}

In the case of non-simply laced algebras, we describe the graph
$\cG_\g$ for algebras of types $B_r,C_r,F_4$ separately from $G_2$.

Let $\g$ be a simple algebra of type $B_r$, $C_4$ or $F_4$. In the
Dynkin diagram of $\g$, let $\lo$ be the long root connected to the
short root $\sh$: 
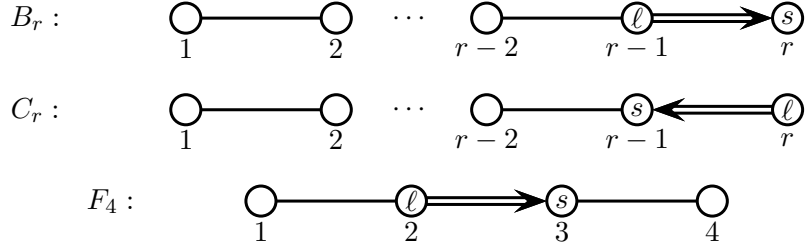
\begin{center}
\psset{unit=2mm,linewidth=.4mm,dimen=middle}
\begin{pspicture}(0,0)(50,5)
\rput(0,2){$B_r:$}
\psline(11,2)(19,2)
\rput(25,2){$\cdots$}
\psline(31,2)(39,2)
\psline[doubleline=true]{->}(41,2)(49,2)
\multips(10,2)(10,0){5}{\pscircle(0,0){1}}
\rput(10,0){1}
\rput(20,0){2}
\rput(30,0){$r-2$}
\rput(40,0){$r-1$}
\rput(50,0){$r$}
\rput(50,2){$s$}
\rput(40,2){$\ell$}
\end{pspicture}
\end{center}
\begin{center}
\psset{unit=2mm,linewidth=.4mm,dimen=middle}
\begin{pspicture}(0,0)(50,5)
\rput(0,2){$C_r:$}
\psline(11,2)(19,2)
\rput(25,2){$\cdots$}
\psline(31,2)(39,2)
\psline[doubleline=true]{<-}(41,2)(49,2)
\multips(10,2)(10,0){5}{\pscircle(0,0){1}}
\rput(10,0){1}
\rput(20,0){2}
\rput(30,0){$r-2$}
\rput(40,0){$r-1$}
\rput(50,0){$r$}
\rput(50,2){$\lo$}
\rput(40,2){$\sh$}
\end{pspicture}
\end{center}
\begin{center}
\psset{unit=2mm,linewidth=.4mm,dimen=middle}
\begin{pspicture}(0,0)(40,5)
\rput(0,2){$F_4:$}
\psline(11,2)(19,2)
\psline[doubleline=true]{->}(21,2)(29,2)
\psline(31,2)(39,2)
\multips(10,2)(10,0){4}{\pscircle(0,0){1}}
\rput(10,0){1}
\rput(20,0){2}
\rput(30,0){3}
\rput(40,0){4}
\rput(20,2){$\lo$}
\rput(30,2){$\sh$}
\end{pspicture}
\end{center}
\vskip.1in
In each case, in the normalized $Q$-system
\eqref{tQsystem} we have
\begin{eqnarray}
\widetilde{\cT}_{2k+1}^{(\sh,\lo)} &=& \tQ_{\lo,k}\tQ_{\lo,k+1} =
\tQ_{\lo,k}^{[-C_{\sh\lo}/2]_+} \tQ_{\lo,k+1}^{[-C_{\sh\lo}/2]_+} 
;\label{odd}\\
\widetilde{\cT}_{2k}^{(\sh,\lo)} &=& \tQ_{\lo,k}^2 = \tQ_{\lo,k}^{[-C_{\sh\lo}]_+}
;\label{even}\\
\widetilde{\cT}_{k}^{(\lo,\sh)} &=& \tQ_{\sh,2k} =
\tQ_{\sh,2k}^{[-C_{\lo,\sh}]_+}\label{shlo}.
\end{eqnarray}
In case $(\al,\beta)$ are not the pair $(\lo,\sh)$ or
$(\sh,\lo)$ then
\begin{equation}\label{otherwise}
\widetilde{\cT}_k^{(\al,\beta)}= \tQ_{\beta,k}^{[-C_{\beta,\al}]_+}.
\end{equation}

Define the sets 
\begin{itemize}
\item $\Pi'=I_r'= \{r+1,...,2r\}$;
\item $\Pi_<$: The subset in $I_r$ corresponding to the short roots of $\g$;
\item $\Pi_>$: The subset in $I_r$ corresponding to the long roots of $\g$;
\item $\Pi_<'=\{\al+r ~|~ \al\in \Pi_<\}$;
\item $\Pi_>'=\{\al+r ~|~ \al\in \Pi_>\}$;
\end{itemize}

\begin{thm}
The graph $\cG_\g$ is obtained from the sequence
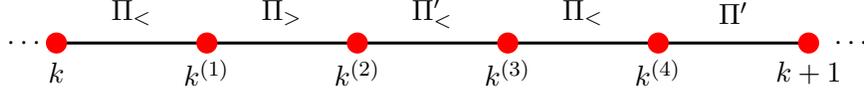
\begin{center}
\psset{unit=2mm,linewidth=.4mm,dimen=middle}
\begin{pspicture}(0,0)(60,10)
\psline(5,5)(55,5)
\multips(5,5)(10,0){6}{\pscircle*[linecolor=red](0,0){.7}}
\rput(5,3){$k$}
\rput(15,3){$k^{(1)}$}
\rput(25,3){$k^{(2)}$}
\rput(35,3){$k^{(3)}$}
\rput(45,3){$k^{(4)}$}
\rput(55,3){$k+1$}
\rput(3,5){$\cdots$}
\rput(58,5){$\cdots$}
\rput(10,7){${\Pi_<}$}
\rput(20,7){${\Pi_>}$}
\rput(30,7){${\Pi_<'}$}
\rput(40,7){${\Pi_<}$}
\rput(50,7){${\Pi'}$}
\end{pspicture}
\end{center}
Here, the cluster at the node labeled $k$ is $\bx[k]$ as in
\eqref{clustervariable} with mutation matrix $B$ as in \eqref{Bmatrix}.
All the compound mutations are products of mutually
commuting mutations, and the cluster variables at each node are in $\cA$.
\end{thm}

To prove this Theorem, we need only list the cluster variables
$(\bx[k^{(i)}],B^{(i)})$ at $k^{(i)}$ -- at intermediate nodes, the
cluster variables are a mixture of cluster variables from $k^{(i-1)}$
and $k^{(i)}$. Note that in this subsection, we read the superscript
of $k$ modulo 5, so $k^{(0)}:=k$ etc. 

We then observe that the relevant columns of $B^{(i)}$ describe
$Q$-system evolutions when applied to the relevant entries of
$\bx[k^{(i)}]$. We use the fact that when $\mu_i$ and $\mu_j$ commute, the
$j$th column of $\mu_i(B)$ is equal to the $j$th column of $B$.

We recall that at the node $k$, we have
$\bx[k] = (x_1,...,x_{2r})$, where
$$
x_\al= \left\{\begin{array}{ll} \tQ_{\al,2t_\al k},& \al\in I_r;\\
\tQ_{\al,2t_\al k + 1}, & \hbox{otherwise}.\end{array}\right.
$$
Moreover,
according to \eqref{Bmatrix}, at the node $k$ we have
$$
B = \left(\begin{array}{cc} A & -C^t\\   C & 0 \end{array}\right).
$$
Here, the submatrix $A$ is no longer the zero matrix. Instead, it has
zeros everywhere
except for two entries, $A_{\lo,\sh}=-1$ and $A_{\sh,\lo}=1$.  Thus,
the mutations $\{\mu_\al |\ \al \in \Pi_<\} $ commute among
themselves, as do the all mutations $\{\mu_{i} |\ i\in \Pi'\}$.

It is clear from the form of the $Q$-system (see Equations
\eqref{odd} and \eqref{otherwise}) that the
mutations $\{ \mu_\al|~\al\in \Pi_<\}$ describe the
evolutions $\mu_\al: \tQ_{\al,2k}\mapsto \tQ_{\al,2k+2}$ ($\al\in
\Pi_<$). This is the same evolution 
as in the simply-laced case if $\al\neq \sh$, and for $\al=\sh$, it follows
from \eqref{odd}.

Thus, edges labeled by $\al\in \Pi_<$ connected to
node $k$ lead to nodes in which all the cluster variables are
in $\cA$. The compound mutation $\mu^{(\Pi_<)}$ leads to the node $k^{(1)}$.

(We also note that the mutations columns $r+1,...,2r$ describe the $Q$-system
evolutions $\tQ_{\al,2k+1}\mapsto \tQ_{\al,2k-1}$, from Equations
\eqref{even}, \eqref{shlo} and \eqref{otherwise}; However due to the
periodic structure of the graph, we do not need to consider this
evolution separately).

\begin{lemma}
\begin{enumerate}
\item The cluster variables at the node $k^{(1)}$ are 
$
\bx[k^{(1)}] = (x_1,...,x_{2r})=\mu^{(\Pi_<)}(\bx[k])$ 
with 
$$x_\al =\left\{ \begin{array}{ll}
 \tQ_{\al,4k+2} & \hbox{ if $\al$ is a short root};\\
 \tQ_{\al,2k}& \hbox{ if $\al$ is a long root};\\ 
 \tQ_{\al,2t_\al k+1} & \hbox{ if $\al>r$};\end{array}\right.
$$
\item The mutation matrix $B^{(1)}$ at $k^{(1)}$ has the following form:
$$
B^{(1)} = \left( \begin{array}{c|c} -A & -D^t \\ \hline D &
    2A\end{array}\right), 
$$
where 
$$D_{ij} = \left\{ \begin{array}{ll}
 C_{ij}, & \hbox{ if $i$ and $j$ are both long roots}; \\
-C_{ij}, & \hbox{ if $j$ is short}; \\
0 & \hbox{otherwise.}\end{array}\right.
$$
\end{enumerate}
\end{lemma}
\begin{proof}
The first statement follows from the discussion above.

The second statement involves several cases.
\begin{itemize}
\item Because the
mutations $\mu_\al$ with $\al<r$ and $\al\in \Pi_<$ mutually commute, 
their effect on the columns and rows of $B$ indexed by $\{\al:~\al\in
\Pi_<\}$ is only to  change their sign. Thus,
$$B_{i,\al}^{(1)}=-B_{\al,i}^{(1)}=-B_{i,\al},\quad \al\in \Pi_<.$$ 
This accounts for the appearance of $-A$ in the upper left hand corner
of $B^{(1)}$, as well as the columns of the matrix $D$ which
correspond to short roots.

\item Let $\beta\in \Pi_>$ be a long root and consider the change 
  column $\beta$ of $B$. If $i\notin \Pi_<$, then
$$
B_{i,\beta}^{(1)} = B_{i,\beta} + \sum_{\al\in \Pi_<}
\sgn(B_{i,\al}) [B_{i,\al}B_{\al,\beta}]_+.
$$
But $B_{\al,\beta}=0$ unless $(\al,\beta)=(\sh,\lo)$. In that case,
$B_{\sh,\lo}=1$. We have that $B_{i,\sh}>0$ if and only if $i=\sh+r$,
in which case it is equal to $2$. Thus, we have
$$
B_{\sh+r,\lo}^{(1)} = 0,
$$
and all other elements in the long-root columns of the lower left hand
corner of $B^{(1)}$ are unchanged. This gives the matrix $D$ in this
quadrant, and $-D^t$ in the upper right block of $B^{(1)}$ since
mutations preserve skew-symmetry.

\item Finally, elements in the lower right quadrant mutate as follows:
\begin{eqnarray*}
B_{i+r,j+r}^{(1)} &=& \sum_{\al\in \Pi_<}
\sgn(B_{i+r,\al})[B_{i+r,\al}B_{\al,j+r}]_+ = 2A.
\end{eqnarray*}
\end{itemize}
\end{proof}

By examination of the matrix $B^{(1)}$ we conclude that at the node
$k^{(1)}$, the mutations $\{\mu_\al:~\al\in \Pi_>\}$ all commute among
themselves.
\begin{lemma}
\begin{enumerate}
\item The cluster variables
  $\bx[k^{(2)}]=(x_1,...,x_{2r})=\mu^{(\Pi_>)}(\bx[k^{(1)}])$ are 
$$
x_{\al} = \left\{ \begin{array}{ll} \tQ_{2 t_\al k + 2} & \al\leq r; \\
\tQ_{2 t_\al k + 1} & \al>r.\end{array}\right.
$$
\item The mutation matrix $B^{(2)}$ has the form
$$
\left(\begin{array}{cc} A & C-A \\ -C^t-A & 2A \end{array}\right).
$$
\end{enumerate}
\end{lemma}
\begin{proof}
Again, by inspection, the columns corresponding to long roots in the
matrix $B^{(1)}$ correspond to the evolution of the $Q$-system
according to Equations \eqref{shlo} and \eqref{otherwise}. This proves
the first statement in the Lemma.

We look at the evolution of the matrix $B^{(1)}$ under $\mu^{(\Pi_>)}$.
\begin{itemize}
\item If $\al\in \Pi_>$ then $B_{i,\al}^{(2)} =
  -B_{i,\al}^{(1)}$. This accounts for the appearance of $A$ in the
  upper left-hand block of $B^{(2)}$ and for the columns
  corresponding to the long roots in the lower left-hand quadrant of
  $B^{(2)}$. 
\item If $\al\in \Pi_<$ and $i>r$ then
$$
B_{i,\al}^{(2)} = B_{i,\al}^{(1)} + \sum_{\beta\in \Pi_>}
\sgn(B_{i,\beta}^{(1)}) [B_{i,\beta}^{(1)}B_{\beta,\al}^{(1)}]_+ 
$$
The summation has a contribution only when $(\beta,\al)=(\lo,\sh)$ in
which case $B_{\lo,\sh}^{(1)}=1$. Then the only non-vanishing
contribution occurs if $i=\beta+r$, in which case
$B_{\beta+r,\beta}^{(1)}=2$ and $B_{\lo+r,\sh} = 3$. This gives the
matrix in the lower left quadrant of $B^{(2)}$.
\item Finally we have
$$
B^{(2)}_{i+r,j+r} = 2A_{i,j} + \sum_{\al\in
  \Pi_>}\sgn(B_{i+r,\al}^{(1)})[B_{i+r,\al}^{(1)} B_{\al,j+r}^{(1)}]_+
= 2A_{ij}.
$$
\end{itemize}
\end{proof}

At this point, it is evident that the mutations $\{\mu_{\al+r}:~\al\in
\Pi_<\}$ commute among themselves, due to the form of the matrix $2A$
appearing in the lower right quadrant of $B^{(2)}$. Moreover, the
corresponding 
columns of $B^{(2)}$ give the $Q$-system evolution for
$\tQ_{\al,2k+1}\mapsto \tQ_{\al,2k+3}$ if $\al$ is
short. 
As a result we have
\begin{lemma}
\begin{enumerate}
\item The cluster variables at $k^{(3)}$ are
  $\mu^{(\Pi_<')}(\bx[k^{(2)}])=\bx[k^{(3)}]=(x_1,...,x_i)$ 
  where
$$
x_{\al+r} = \tQ_{\al,4k+3},\qquad \al\in \Pi_<,
$$
and the other cluster variables are unchanged from $\bx[k^{(2)}]$.
\item The mutation matrix $B^{(3)}=-B^{(1)}$
\end{enumerate}
\end{lemma}
The proof is by a calculation analogous to the previous two lemmas. 

The columns $\{\al:~\al\in \Pi_<\}$ of $-B^{(1)}$ describe the
mutations corresponding to the evolutions
$\tQ_{\al,4k+2}\to\tQ_{\al,4k+4}$ as they are identical to the corresponding
columns in the matrix $B$. The mutations corresponding to the short
roots again commute. We have the compound evolution
$\mu^{(\Pi_<)}$
which brings us to the node $k^{(4)}$, with
\begin{lemma}
\begin{enumerate}
\item The cluster variables at the point $k^{(4)}$ are
  $\mu^{(\Pi_<)}(\bx[k^{(3)}])=\bx[k^{(4)}]=(x_1,...,x_{2r})$, 
  with
$$
x_\al = \left\{ \begin{array}{ll} \tQ_{\al,2t_\al (k + 2)}, & \al\leq r \\
\tQ_{\al, 2 k + 1}, & \al-r\in \Pi_> \\
\tQ_{\al, 2k+3}, & \al-r \in \Pi_<. \end{array}\right.
$$
\item The mutation matrix at the node $k^{(4)}$ is $B^{(4)}=-B.$
\end{enumerate}
\end{lemma}
Again, the proof is analogous to that of the previous lemmas.

Finally we note that at the node $k^{(4)}$, not only do the mutations
$\{\mu_{\al+r}:\al\in \Pi\}$ all commute among themselves, the
corresponding columns of $B^{(4)}$ describe the evolution of the
$Q$-system for any $\tQ_{\al,i}$ with $i$ odd. Therefore the
compound evolution $\mu^{(\Pi')}$ makes sense at this node.  The node
$k^{(5)}$ is reached from node $k^{(4)}$ by acting with
$\mu^{(\Pi')}$. In fact, we find that the node $k^{(5)}$ is the node
$k+1$ in $\cG_\g$:
\begin{lemma}
\begin{enumerate}
\item The cluster variables $\bx[k^{(5)}]=\mu^{(\Pi')}(\bx[k^{(4)}]) =
  \bx[k+1]$,
  where $\bx[k]$ for any $k\in\Z$ was defined in equation
  \eqref{clustervariable}. 
\item The mutation matrix at the node $k^{(5)}$ is $B$.
\end{enumerate}
\end{lemma}
The proof is by direct calculation. Thus, there is a periodicity to the
graph $\cG_\g$, where the matrix $B^{(i)}$ is repeated along the
``band'' corresponding to the subgraph every fifth step.

\subsection{The $Q$-system for $G_2$ as a cluster algebra}
In this case the normalized $Q$-system \eqref{tQsystem} has terms
$\widetilde{\cT}_{\al,k}$ as follows:
\begin{eqnarray*}
\tcT_{k}^{(1,2)} &=& \tQ_{2,3k} \\
\tcT_{3k}^{(2,1)} &=& \tQ_{1,k}^3 \\
\tcT_{3k+1}^{(2,1)} &=& \tQ_{1,k}^2 \tQ_{1,k+1} \\
\tcT_{3k+2}^{(2,1)} &=& \tQ_{1,k}\tQ_{1,k+1}^2.
\end{eqnarray*}

Again we will describe a graph $\cG_{G_2}$, a quotient graph of
a subgraph of $\mathbb T_4$, which contains at its nodes cluster
variables with all the normalized characters $\{\tQ_{\al,k}:~\al\in
I_2,~ k\in\Z_+\}$.

The graph $\cG_{G_2}$ contains all nodes labeled by $k\in \Z_+$, and
also the nodes $k^{(i)}$ with $i=1,...,6$. At these nodes we have
cluster variables $\bx[k^{(k)}]$ and mutation matrices $B^{(i)}$.
We will describe the
clusters at those nodes explicitly. As in the other cases, we choose
to evolve along nodes of the graph when the corresponding mutations
describe one of the $Q$-system equations. The graph $\cG_{G_2}$ is the
following graph:

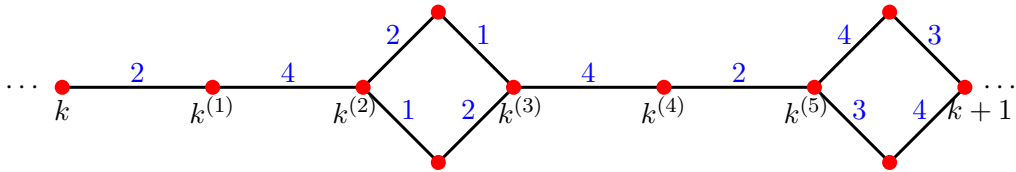
\begin{center}
\psset{unit=1mm,linewidth=.4mm,dimen=middle}
\begin{pspicture}(0,0)(120,20)
\rput(-5,10){$\cdots$}
\rput(125,10){$\cdots$}
\multips(0,0)(20,0){2}{\psline(0,10)(20,10)}
\psline(40,10)(50,20)(60,10)(50,0)(40,10)
\psline(60,10)(100,10)
\psline(100,10)(110,20)(120,10)(110,0)(100,10)
\multips(0,0)(20,0){7}{\pscircle*[linecolor=red](0,10){1}}
\pscircle*[linecolor=red](50,20){1}
\pscircle*[linecolor=red](50,0){1}
\pscircle*[linecolor=red](110,20){1}
\pscircle*[linecolor=red](110,0){1}
\rput(0,7){$k$}
\rput(20,7){$k^{(1)}$}
\rput(39,7){$k^{(2)}$}
\rput(61,7){$k^{(3)}$}
\rput(80,7){$k^{(4)}$}
\rput(99,7){$k^{(5)}$}
\rput(122,7){$k+1$}
\rput(10,12){${\blue 2}$}
\rput(30,12){${\blue 4}$}
\rput(44,17){${\blue 2}$}
\rput(46,7){${\blue 1}$}
\rput(56,17){${\blue 1}$}
\rput(54,7){${\blue 2}$}
\rput(70,12){${\blue 4}$}
\rput(90,12){${\blue 2}$}
\rput(106,7){${\blue 3}$}
\rput(104,17){${\blue 4}$}
\rput(116,17){${\blue 3}$}
\rput(114,7){${\blue 4}$}
\end{pspicture}
\end{center}

With the notation of the previous subsection, we can also depict this
graph as follows:
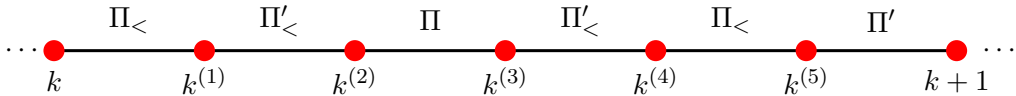
\begin{center}
\psset{unit=2mm,linewidth=.4mm,dimen=middle}
\begin{pspicture}(0,0)(70,10)
\psline(5,5)(65,5)
\multips(5,5)(10,0){7}{\pscircle*[linecolor=red](0,0){.7}}
\rput(5,3){$k$}
\rput(15,3){$k^{(1)}$}
\rput(25,3){$k^{(2)}$}
\rput(35,3){$k^{(3)}$}
\rput(45,3){$k^{(4)}$}
\rput(55,3){$k^{(5)}$}
\rput(65,3){$k+1$}
\rput(3,5){$\cdots$}
\rput(68,5){$\cdots$}
\rput(10,7){${\Pi_<}$}
\rput(20,7){${\Pi_<'}$}
\rput(30,7){${\Pi}$}
\rput(40,7){${\Pi_<'}$}
\rput(50,7){${\Pi_<}$}
\rput(60,7){${\Pi'}$}
\end{pspicture}
\end{center}

The following can be easily verified by explicit calculation:
\begin{thm}
The mutations along the graph $\cG_{G_2}$ describe
the evolution of the $Q$-system of $G_2$, the
cluster variables at each point $k^{(i)}$ are
defined as follows:
\begin{eqnarray*}
\begin{array}{|r|c|c|}\hline
i & \bx[k^{(i)}] & B^{(i)} \\ \hline\hline
0 & (\tQ_{1,2k},\tQ_{2,6k},\tQ_{1,2k+1},\tQ_{2,6k+1}) &
 \left( \begin{array}{rrrr}
0 & -2 & -2 & 3 \\
2 & 0  & 1 & -2 \\
2 & -1 & 0 & 0 \\
-3 & 2 & 0 & 0
\end{array}\right) \\ \hline
1 & (\tQ_{1,2k},\tQ_{2,6k+2},\tQ_{1,2k+1},\tQ_{2,6k+1})&
\left( \begin{array}{rrrr}
0 & 2 & -2 & -1 \\
-2 & 0  & -1 & 2 \\
2 & 1 & 0 & -2 \\
1 & -2 & 2 & 0
\end{array}\right)
\\ \hline
2 & (\tQ_{1,2k},\tQ_{2,6k+2},\tQ_{1,2k+1},\tQ_{2,6k+3})&
 \left( \begin{array}{rr}
0 & -C \\ 
C^t & -A 
\end{array}\right)\\ \hline
3 & (\tQ_{1,2k+2},\tQ_{2,6k+4},\tQ_{1,2k+1},\tQ_{2,6k+3})&
-B^{(2)} \\ \hline
4 & (\tQ_{1,2k+2},\tQ_{2,6k+4},\tQ_{1,2k+1},\tQ_{2,6k+5})&
-B^{(1)} \\ \hline
5 & (\tQ_{1,2k+2},\tQ_{2,6k+6},\tQ_{1,2k+1},\tQ_{2,6k+5})&
 -B \\ \hline
6 & \bx[k+1]=(\tQ_{1,2k+2},\tQ_{2,6k+6},\tQ_{1,2k+3},\tQ_{2,6k+7})&
B \\ \hline
\end{array}
\end{eqnarray*}
\end{thm}
Note that $k^{(0)}:=k$ and $k^{(6)}=k+1$ as the cluster variables at
that node correspond to the point $k+1$.

To prove this theorem, one simply compares the $Q$-system evolution
along in the direction of the label of the edge emanating from the
node $k^{(i)}$, and the corresponding columns of the mutation matrices
at the points $k^{(i)}$. That is, if the an edge labeled $n$ emanates
from $k^{(i)}$ in the graph $\cG_{G_2}$, then the
$n$th variable of $x[k^{(i)}]$ satisfies the $Q$-system described by
column $n$ of $B^{(i)}$. The ``diamonds'' in $\cG_{G_2}$ occur
when two mutations commute. Since this is a finite system it can be done
by a direct calculation.

\section{Polynomiality phenomenon}\label{poly}
We now consider the specialization of the cluster algebra to the point
$\{Q_{\al,0}=1\}_\al$ (so that $\tQ_{\al,0}=\epsilon_\al$), which 
we call the Kirillov-Reshetikhin point, because solutions to the
$Q$-system with this boundary condition, for $m>0$, are characters of
Kirillov-Reshetikhin modules.

Note that at the KR point, $Q_{\al,-1}=\tQ_{\al,-1}=0$ for all $\al$.
However due to the Laurent phenomenon, which states, among other
things, that all cluster variables are Laurent polynomials in
$\tQ_{\al,0}$ and $\tQ_{\al,1}$, all cluster variables are
well-defined at this point.  They have no singularity since at the KR
point $\tQ_{\al,0}\neq 0$.

In our proof of the combinatorial Kirillov-Reshetikhin conjecture
\cite{DFK}, a crucial ingredient was the following Lemma:
\begin{lemma}\label{polynomiality}
For any algebra $\g$, any solution to the KR Q-system $Q_{\al,m}$ is a
polynomial in the $r$ variables $\{Q_{\al,1}\}$.
\end{lemma}
This follows from the fact that KR-modules are all in the
Groethendieck group of the trivial and fundamental KR-modules.

One of the reasons for recasting the $Q$-system as a cluster algebra
is the availability of the Laurent phenomenon theorem. It implies a
purely algebraic proof for the polynomiality lemma. (The idea for this
proof is due to Sergei Fomin.) We will explain the general situation
which gives rise to this property in the $Q$-system.

Consider a cluster algebra with a special node, which we call $0$ or
the origin,
with cluster variables 
$$\bx = (a_1,...,a_n;b_1,...,b_m)=(\ba,\bb)$$ and mutation
matrix $B$ such that 
$$B_{ij}=0\qquad \hbox{ if $1\leq i,j\leq n$.}$$

We note that, as a result of this assumption about the form of the
$B$-matrix, we have
$$
\mu_i(a_i) = \frac{N_i(\bb)}{a_i},\qquad 1\leq i \leq n,
$$
where
$$
N_i(\bb) = \prod_{j} b_j^{[B_{j+n,i}]_+} +\prod_{j} b_j^{[-B_{j+n,i}]_+}
$$
is a function of $\{b_1,...,b_m\}$ only, because $B_{j,i}=0$ if
$i,j\leq n$.

We assume that there is a special point in the parameter space $\bb$,
which we call the KR point in analogy with the $Q$-system situation,
such that for all $i\leq n$, $\mu_i(a_i)=0$.  Evaluation of $\bb$ at
the KR point means the evaluation at the point $N_i(\bb)=0$ ($i\leq n$).
This point depends only on $\bb$, not on $\ba$, which are left as
formal variables.

We will show that for a cluster variable $y(\ba,\bb)$ at any node of
the cluster tree $\mathbb T_{m+n}$, considered as a function of
$(\ba,\bb)$ (a Laurent polynomial, due to the Laurent property), terms
with singularities at any $a_i=0$ vanish at the KR point.

We use the vector notation, so that if $\mathbf b = (b_1,...,b_i)$ and
$\bm = (m_1,...,m_i)$ then $\mathbf b^\bm=\prod_j b_j^{m_j}$.

\begin{lemma}
Let
$$
y(\ba,\bb) = \bb^{-\bm}P(\ba,\bb)
$$
where $P(\ba,\bb)$ is a polynomial in $\bb$ but may have some
singularities in $\{a_i\}$. That is,
$$
P(\ba,\bb) = \sum_{\bn\in\Z^r} C_\bn(\bb) \ba^\bn,
$$
with only finitely many non-zero terms, and where $C_n(\bb)$ are
polynomials in the variables $\bb$. Then $C_\bn(\bb)$ is divisible by
$$
\prod_{j: n_j<0} N_j(\bb)^{-n_j}.
$$
\end{lemma}
\begin{proof}
By change of variables to $\mu_i(a_i)=a_i'$ for all $i\leq n$, we have that 
$$
P(\ba,\bb) = \sum_{\bn\in \Z^r} C_\bn(\bb) \prod_{j=1}^n
(a'_j)^{-n_j} N_j(\bb)^{n_j}.
$$
Any terms with $n_j<0$ have a positive power of $N_j(\bb)$ in the
denominator. These must cancel with factors in $C_\bn(\bb)$, for each
$\bn$ which has negative components, because the
Laurent property states that a cluster variable cannot be a rational
function in the 
variables $\bb$. Furthermore, there can be no cancellations between terms with
different $\bn$, since the $a'_j$ are independent variables. 
Thus, $C_\bn(\bb)$ is divisible by $N_j(\bb)^{-n_j}$ for each negative
$n_j$. 
\end{proof}

\begin{remark}
It was pointed out to us after this paper was completed that this
Lemma follows from Lemma 4.2 of \cite{BFZ}.
\end{remark}

Thus we have what we called ``strong Laurent phenomenon'' in
\cite{Ke07} or the polynomiality property:
\begin{cor}
With the cluster algebra as defined above, at the evaluation of $\bb$
at a point such that $\mu_i(a_i)=N_i(\bb)=0$,
all cluster variables are polynomials in the variables $\{a_i\}$.
\end{cor}
\begin{proof}
$C_\bn(\bb)$ is divisible by $N_j(\bb)^{-n_j}$ for each negative
$n_j$. Therefore, it vanishes at the KR point, where $N_j(\bb)=0$ for
each $j$. These are all the terms which have poles at $a_j=0$. 
\end{proof}

Lemma \ref{polynomiality} is a special case of this statement (with a
rearrangement of the indices). We have $m=n=r$, $a_i=Q_{i,1}$ and
$b_i=Q_{i,0}$. The indices of the matrix $B$ are accordingly
rearranged. Note that, in all the algebras introduced so far, we had
the property that $B_{ij}=0$ if $r<i,j\leq 2r$, which is the analog of
the condition on $B$ above.

We conclude that the Laurent phenomenon is responsible for the
polynomiality property of the characters $Q_{\al,i}$ in terms of
$Q_{\al,1}$ for the KR $Q$-system. This property is quite general and
is not related to the appearance of a Cartan matrix in the mutation
matrix, only on the block structure of the mutation matrix, and the
specific nature of the KR point in the parameter space.

Moreover, we can conclude much more.
\begin{cor}
Any cluster variable at any node
of the full cluster graph $\mathbb T_{2r}$ extending $\cG_\g$ as
defined in Section \ref{qascluster} is a polynomial in $\tQ_{\al,1}$ after
specialization of the parameters to the $KR$ point,
$\tQ_{\al,0}=\epsilon_\al$.  
\end{cor}

We do not, at this point, have a
representation-theoretical interpretation for this
phenomenon.

\section{Discussion}\label{discussion}
\subsection{Deformed Q-systems and the cluster algebra}

In \cite{DFK}, we have defined deformed Q-systems depending 
on infinite families of parameters
$\bu=\{u_\al,u_{\al,1},u_{\al,2},...\}_{\al\in I_r}$, 
as well as the extra parameters
$\ba = \{ a_1,a_2,...\}$ in the non-simply-laced case. The deformed
$Q$-system is defined as follows:
\begin{equation}\label{cqsys}
u_{\al,n}\cQ_{\al,n+1}\cQ_{\al,n-1}=\cQ_{\al,n}^2 - \prod_{\beta \sim
  \al} {\mathcal U}_{n}^{(\al,\beta)}, \qquad \al\in I_r,\ n\geq 1,
\end{equation}
where 
$${\mathcal
  U}_{n}^{(\al,\beta)}=\prod_{i=0}^{|C_{\al,\beta}|-1}
\cQ_{\beta,\lfloor\frac{t_\beta n+i}{t_\al} \rfloor},\qquad
(\al,\beta)\neq (\sh,\lo)$$
and
$${\mathcal
  U}_{n}^{(\al,\beta)}=a_n^{t_\al(1+\lfloor\frac{n}{t_\al}\rfloor)-n}
\prod_{i=0}^{|C_{\al,\beta}|-1} \cQ_{\beta,\lfloor\frac{n+i}{t_\al}
  \rfloor}\qquad (\al,\beta)=(\sh,\lo).$$

The deformed $Q$-system \eqref{cqsys} is subject to the  initial
conditions that $\cQ_{\al,0}=1$ and $\cQ_{\al,1}=u_\al^{-1}$, for all
$\al \in I_r$.

The deformed variables $\cQ(\bu,\ba)$
were introduced as a tool for proving the $M=N$ identity of \cite{HKOTY},
yielding fermionic  expressions for the  tensor product
multiplicities of KR-modules.

We may view the cluster variables $\tQ_{\al,k}$ of the cluster graph
$\cG_\g$ of Section \ref{qascluster} as specializations of the
deformed characters $\cQ(\bu,\ba)$ (after normalization). 
This amounts to relaxing the initial condition on the $Q$-system, which
is the specialization to the KR point.

Let $\varphi$
denote the evaluation
$$\varphi(a_i)=1,\quad \varphi(u_{\al,i})=1,\quad i\geq 2,$$ 
and
$$
\varphi(a_1)=u_{\sh,1}$$ 
where $\sh$ the unique short root connected to a long root. 

The evaluated solutions of the 
deformed $Q$-system
$$\overline{\cQ}(\{u_\al,u_{\al,1}\}_{\al\in
  I_r}):=\varphi(\cQ(\bu,\ba))$$ 
obey a system of equations identical
to the Q-system \eqref{Qsystem} for all $k\geq 1$, except
for the difference that $\overline{\cQ}_{\al,0}$ always comes with a
prefactor $u_{\al,1}$. 

Therefore, we may absorb the parameter
$u_{\al,1}$ into the definition of $\overline{\cQ}_{\al,0}$, namely by
introducing 
$$Q_{\al,k}(\{u_{\beta,1},u_\beta^{-1}\}_{\beta\in
  I_r}):=u_{\al,1}^{\delta_{k,0}}
\overline{\cQ}_{\al,k}(\{u_\beta,u_{\beta,1}\}_{\beta\in I_r}).$$
  Then
$Q_{\al,k}$ obey the Q-system \eqref{Qsystem} for $k\geq 1$, with
the initial conditions replaced by
$Q_{\al,0}=u_{\al,1}$ and $Q_{\al,1}=u_\al^{-1}$. This is identical to
\eqref{deformedinitial}. 

Upon the renormalization of the variables as in Lemma
\ref{renormalize}, we deduce that the specialized deformed KR
characters are related to the variables $\tQ_{\al,k}$ via:
\begin{equation}
  \tQ_{\al,k}=\epsilon_\al u_{\al,1}^{\delta_{k,0}}
  \varphi\Big(\cQ_{\al,k}(\bu,\ba)\Big).
\end{equation}

We note that without the evaluation $\varphi$, the variables
$\cQ_{\al,m}(\bu,\ba)$ are not obviously part of a cluster
algebra. For example, they do not have the Laurent property.

In \cite{DFK}, we derived and used a crucial ``substitution
invariance'' property of the deformed KR characters $\cQ$ which, after
the specialization $\varphi$, reads as follows for the $Q$'s:
\begin{equation}\label{translaQ}
Q_{\al,k+t_\al
  j}(\{Q_{\beta,0},Q_{\beta,1}\})=Q_{\al,k}(\{Q_{\beta,t_\beta
  j},Q_{\beta,t_\beta j+1}\}) .
\end{equation}
We showed that this condition is equivalent to the $Q$-system, with
appropriate initial conditions.

When rephrased in terms of $\tQ$'s, with $j=2m$ even, this is
nothing but the property of translational invariance of the cluster
graph $\cG_\g$. Namely, the expression of cluster variables at any
node $n$ of $\cG_\g$ in terms of that at cluster variables at the node $0$ 
is the same as
that of the cluster variables at the node $n+m$ in terms of the
cluster variables at the node $m$. This is obvious from the structure
of the graph.

Equation \eqref{translaQ} is actually slightly more general, as it also
includes translations by an odd integer $j$. For instance, the case
$j=1$ corresponds in the cluster formulation to the
$k\to k'$ (``half-'') translation in the simply-laced case, and to $k\to
k^{(3)}$ in the $G_2$ case. 
At the node $k'$ (or
$k^{(3)}$ for $G_2$), the mutation matrix is indeed $JBJ$, where
$J=\left(\begin{matrix} 0 & I_{r\times r} \\ I_{r\times r} & 0
  \end{matrix}\right)$, hence the nodes $k$ and $k'$ (or $k$ and
$k^{(3)}$ for $G_2$) are equivalent upon interchanging the even and
odd cluster variables $\bx \to J \bx$, and the half-translation
invariance property \eqref{translaQ} for odd $j$ follows.

\subsection{Extended graphs}\label{extensions}

As remarked above, generally there are larger subgraphs of $\mathbb
T_{2r}$ which correspond to $Q$-system evolutions. 

\subsubsection{Graphs corresponding to $Q$-system evolutions}
In Example \ref{Athree} for $A_3$, we did not describe edges which do
not correspond to $Q$-system evolutions. But in addition this graph
does not describe all edges which are $Q$-system evolutions. For
example, $\mu_6$ acting on $\mu_2\circ \mu_3(\bx[k],B)$ is a
$Q$-system evolution, which results in the cluster variable
$$\bx=(\tQ_{1,2k},\tQ_{2,2k+2},\tQ_{3,2k+2},\tQ_{1,2k+1},
\tQ_{2,2k+1},\tQ_{3,2k+3}).$$ 
Similarly, $\mu_4$ acting on $\mu_1\circ \mu_2(\bx[k],B)$ is a
$Q$-system evolution.

 We can describe the complete graph corresponding
  to $Q$-system evolutions in the case of $A_3$:
\vskip.1in
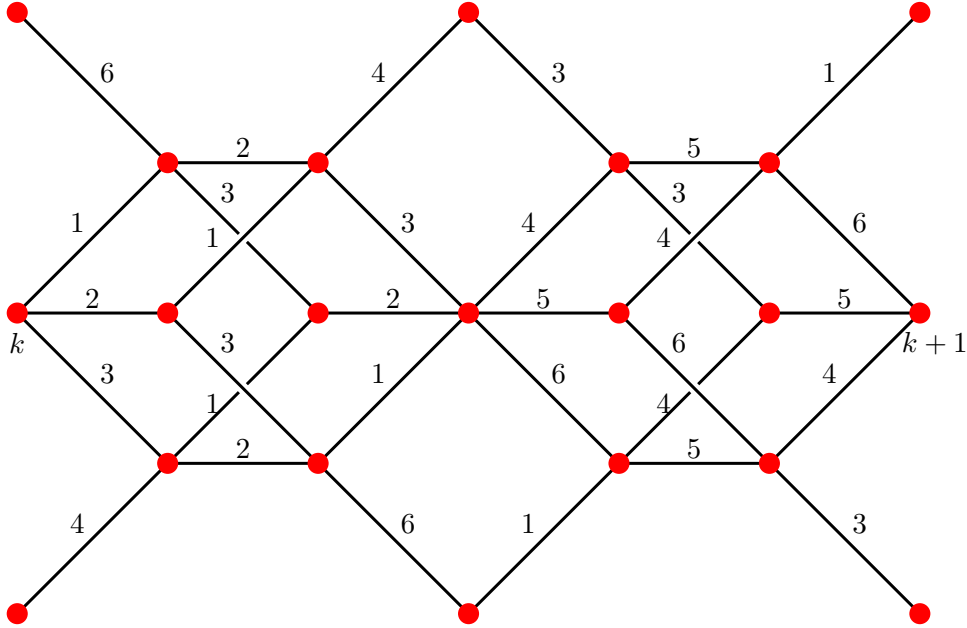
\begin{center}
\psset{unit=2mm,linewidth=.4mm,dimen=middle}

\begin{pspicture}(0,0)(60,40)
\psline(0,0)(10,10)(0,20)(10,30)(0,40)
\psline(10,10)(20,20)(10,30)
\psline[linewidth=1.5mm,linecolor=white](20,10)(10,20)(20,30)
\psline(20,10)(10,20)(20,30)
\psline(10,10)(20,10)
\psline(20,30)(10,30)
\psline(0,20)(10,20)
\psline(30,0)(20,10)(30,20)(20,30)(30,40)(40,30)(30,20)(40,10)(30,0)
\psline(40,10)(50,20)(40,30)
\psline[linewidth=1.5mm,linecolor=white](50,10)(40,20)(50,30)
\psline(50,10)(40,20)(50,30)
\psline(20,20)(40,20)
\psline(60,0)(50,10)(60,20)(50,30)(60,40)
\psline(40,10)(50,10)
\psline(50,20)(60,20)
\psline(40,30)(50,30)
\multips(0,0)(30,0){3}{\pscircle*[linecolor=red](0,0){.7}}
\multips(0,40)(30,0){3}{\pscircle*[linecolor=red](0,0){.7}}
\multips(0,20)(10,0){7}{\pscircle*[linecolor=red](0,0){.7}}
\multips(10,10)(10,0){2}{\pscircle*[linecolor=red](0,0){.7}}
\multips(40,10)(10,0){2}{\pscircle*[linecolor=red](0,0){.7}}
\multips(10,30)(10,0){2}{\pscircle*[linecolor=red](0,0){.7}}
\multips(40,30)(10,0){2}{\pscircle*[linecolor=red](0,0){.7}}
\rput(4,6){4}
\rput(6,16){3}
\rput(5,21){2}
\rput(4,26){1}
\rput(6,36){6}
\rput(15,11){2}
\rput(15,31){2}
\rput(13,14){1}
\rput(13,25){1}
\rput(14,18){3}
\rput(14,28){3}
\rput(24,36){4}
\rput(26,26){3}
\rput(25,21){2}
\rput(24,16){1}
\rput(26,6){6}
\rput(34,6){1}
\rput(36,16){6}
\rput(35,21){5}
\rput(34,26){4}
\rput(36,36){3}
\rput(45,11){5}
\rput(45,31){5}
\rput(43,14){4}
\rput(43,25){4}
\rput(44,18){6}
\rput(44,28){3}
\rput(54,36){1}
\rput(56,26){6}
\rput(55,21){5}
\rput(54,16){4}
\rput(56,6){3}
\rput(0,18){$k$}
\rput(61,18){$k+1$}
\end{pspicture}
\end{center}

\subsection{Generalizations and open questions}
As mentioned in the introduction, there are generalizations of
$Q$-systems to Cartan matrices corresponding to affine algebras
\cite{Her07}. Moreover there are generalizations of $Q$-systems to
twisted quantum affine algebras. It is reasonable to expect that these
can be recast in the cluster algebra language, although Lemma
\ref{renormalize} does not, in general, work for affine Cartan
matrices. 

$Q$-systems are obtained from $T$-systems by taking the combinatorial
limit, which amounts to ignoring the dependence on the spectral
parameter \cite{KR,KunibaYsys}. The $T$-system is a discrete Hirota
equation in the case of $A_n$, and in general has an integrability
property. As a result, $Q$-systems inherit some remarkable properties,
such as the existence of conserved quantities and associated linear
systems. We will address this property in a future publication.

An interesting question arises from the $M=N$ identity which was
proved in \cite{DFK}. This is a combinatorial identity which relates a
restricted sum over products of binomial coefficients to an
unrestricted one (see \cite{HKOTY} for the original conjecture). 

For example, the $M=N$ identity has the following form for $A_1$:
$$
\sum_{\bm\in \Z_+^k~:~p_i\geq 0} \prod_i {m_i+p_i \choose m_i} =
\sum_{\bm\in \Z_+^k} \prod_i {m_i+p_i \choose m_i}
$$
where for a choice of non-negative integers $\bn\in \Z_+^k$, $p_i = 
\sum_j\min(i,j) (n_j - 2 m_j)$ and the sum is taken over $\bm\in
\Z_+^k$ such that $\sum_j j(n_j - 2 m_j)=l\in \Z_+$.

The restricted sum on the left hand side is manifestly positive,
whereas the unrestricted sum on the right is an alternating sum. In
\cite{DFK}, we have recast
this identity in terms of generating functions and their power series
expansions in order to prove the original conjecture of
\cite{HKOTY}. The $M=N$
identity is a property of the following generating function:
$$
Z_{l;\bn}^{(k)}(\bu)=
\frac{\cQ_1(\bu)\cQ_k(\bu)^{l+1}}{\cQ_{k+1}(\bu)^{l+1}}\prod_{i=1}^k
\frac{\cQ_i(\bu)^{n_i}}{u_i},
$$
where $\cQ_m(\bu)$ are solutions of the deformed $Q$-system.  The
$M=N$ identity is the following statement: The constant term of $Z$ in
$u$, after evaluation at $u_1=\cdots=u_k=1$, is equal to the constant
term in $u$ of the power series part of $Z$ in each of the $u_i$,
after evaluation at the same point.

Our proof relies only on the polynomiality property of the
solutions to the KR $Q$-system.  It would be interesting to understand
this identity more fully in the context of the properties of the
associated cluster algebras.

\begin{appendix}
\newcommand{\bq}{\mathbf q}
\newcommand{\even}{{\rm even}}
\newcommand{\odd}{{\rm odd}}
\newcommand{\sign}{{\rm sign}}

\section{$Q$-systems with coefficients}
In this paper, we considered only $Q$-systems which correspond to some
finite-type Cartan matrix. For this reason, it is always possible to
renormalize the variables $Q_{\alpha,k}$ which appear in the
$Q$-system \eqref{Qsystem} so as to eliminate the minus sign appearing on
the right hand side, due to Lemma \ref{renormalize}. 
This is necessary in order to conform with the
usual definition of a cluster mutation, which is a subtraction-free
expression. 

Such renormalization is not always possible, for example for certain
affine-type Cartan matrices, and therefore in general, it is
preferable to 
reformulate the system in terms of a cluster algebra with coefficients
\cite{FZfour}, in such a way that a specialization of the values of
the coefficients reproduces the original $Q$-system \eqref{Qsystem}.

It is possible to formulate this by introducing the coefficients
$\{q_1,...,q_r\}$ as extra cluster variables {\em which do not
mutate}, and then writing a ``$Q$-system with coefficients'' by
replacing the minus sign on the right hand side of Equation
\eqref{Qsystem} with these coefficients. 

The resulting $Q$-systems with
coefficients are 
\begin{equation}\label{Qsyscoeff}
Q_{\al,k+1}Q_{\al,k-1}=Q_{\al,k}^2 + q_\al \prod_{\beta\sim\al}
\mathcal T_k^{(\alpha,\beta)},
\end{equation}
with $\mathcal T_k^{(\alpha,\beta)}$ as in equation \eqref{T}. When
$q_1=\cdots=q_r=-1$, this reproduces the original $Q$-system \eqref{Qsystem}.

For each simple Lie algebra, the systems \eqref{Qsyscoeff} correspond
again to a cluster algebra on the subgraphs $\mathcal G_r$ introduced
in Section 3.  Each mutation along edges of $\mathcal G_r$ is one of
the equations in the system of equations \eqref{Qsyscoeff}.  In this
case, the
coefficients $q_\alpha$ enter the exchange relation via an augmented
exchange matrix $\widetilde{B}$. The formulation is as follows.

At the node $k$ (in the notation of Section 3) of the graph $\mathcal
 G_r$, we have the augmented cluster variables 
$$\bx=(\mathbf Q_{{\rm even}},\mathbf Q_{{\rm odd}}; q_1,...,q_r),$$
 where $\mathbf Q_{{\rm even}}$ is the collection of elements
 $(Q_{\alpha, 2 t_\alpha k})_{\alpha\in I_r}$ and $\mathbf Q_{\rm
 odd}=(Q_{\alpha, 2 t_\alpha k + 1})_{\al\in I_r}$. Thus, $\bx$ is a
 collection of $3r$ elements, the last $r$ of which are the
 coefficients $q_1,...,q_r$, which do not mutate.

The exchange matrix is, in principle, of size $3r\times 3r$ but we
need only specify the first $2r$ columns, because the coefficients
$\{q_\alpha\}_{\alpha\in I_r}$, which are the last $r$ variables of
$\bx$, do not mutate.  The relevant $3r\times 2r$ exchange matrix,
consisting of the first $2r$ columns, is denoted by $\widetilde{B}$
and we refer to it as the {\em augmented exchange matrix}.

\begin{defn}
The augmented exchange matrix $\widetilde B$ at node $k$ is a
$3r\times 2r$-matrix with entries as follows: The first
$2r$ rows of $\widetilde{B}$ coincide with those of $B$, while the
last $r$ rows have entries
$$
\widetilde{B}_{2r+\al,\beta}=-\delta_{\al,\beta}=
-\widetilde{B}_{2r+\al,r+\beta},\qquad  
\al,\beta\in I_r.
$$
\end{defn}

We claim that the cluster mutations 
\begin{equation}\label{exchange}
\mu_i:~ x_i \mapsto x_i^{-1}\left( 
\prod_{1\leq j\leq 3r} x_j^{[\widetilde{B}_{ji}]_+} +
  \prod_{1\leq j\leq 3r}  
  x_j^{[-\widetilde{B}_{ji}]_+}\right), \quad 1\leq i \leq 2r.
\end{equation}
along the edges of the cluster
subgraph $\mathcal G_r$ coincide with the relations \eqref{Qsyscoeff}.
In fact we have already shown this for the submatrix $B$ in Section 3.
We need only show that the augmentation of $\widetilde{B}$ has the
correct evolution.
This can be
seen either algebraically or graphically. 

\begin{example}
Let us illustrate this structure in the simply-laced case. In this case,
at node $k$, the augmented exchange matrix has the block form
$$ \widetilde{B} = \left[\begin{array}{cc} 0 & -C \\ C & 0 \\ -I & I
\end{array} \right] ,
$$
where $I$ is the $r\times r$ identity matrix and $C$ is the Cartan matrix.

In the subgraph $\mathcal G_r$, mutations along
the edges labeled $1,...,r$ are $Q$-system evolutions corresponding to
the evolution of variables with even indices. These commute, and since
$\widetilde{B}_{2r+\alpha,\beta}=-\delta_{\alpha,\beta}$, the
coefficient $q_\alpha$ is introduced in the second term
of the right hand side of \eqref{exchange}.

Thus, the mutation
$\mu_\alpha$ with $1\leq \alpha\leq r$ has the effect of changing the
sign of $\widetilde{B}_{2r+\alpha,\alpha}$ and leaving the rest of the
block $(\widetilde{B}_{ij})_{2r<i\leq 3r \atop 1\leq j\leq r}$ unchanged. 

It was shown in Section 3 that 
$\mu_{1}\circ \cdots \circ\mu_r(B)=-B$.

Moreover, the block $(B_{ij})_{2r<i\leq 3r \atop r<j\leq 2r}$ also
changes sign, because
\begin{eqnarray*}
\mu_1\circ \cdots \circ
\mu_r(B_{2r+\alpha,r+\beta})&=&B_{2r+\alpha,r+\beta}+ \sum_{\gamma=1}^r 
B_{2r+\alpha,\gamma} [B_{2r+\alpha,\gamma}B_{\gamma,r+\beta}]_+\\
&=& B_{2r+\alpha,r+\beta}+ \sum_{\gamma=1}^r (-\delta_{\alpha,\gamma})
2\delta_{\gamma,\beta}\\
& =& -\delta_{\alpha,\beta}=-B_{2r+\alpha,r+\beta}.
\end{eqnarray*}
Thus, we have that $\mu_1\circ \cdots \circ
\mu_r(\widetilde{B})=-\widetilde{B}$ at the node $k'$.

Similarly, the effect of acting with the mutations $\mu_{r+1},... ,\mu_{2r}$
on $\widetilde{B}'$ again changes its overall sign.
\end{example}

In general, it is easiest to illustrate this part of the evolution of
$\widetilde{B}$ graphically. We represent a
skew-symmetric integer matrix as a quiver, where the
nodes are enumerated by the rows of the matrix, and if $A_{ij}=m>0$
then the node $i$ is connected to the node $j$ by an arrow from $i$ to
$j$ labeled by $m$.

In the case of the cluster algebra without coefficients, the quivers
have nodes $\alpha\in \{1,...,r\}$, corresponding to variables
$\mathbf Q_{\rm even}$
with even indices, and $\{\overline{\alpha}=r+\alpha, \alpha\in I_r\}$
corresponding to the variables with odd indices. Their connectivity is
determined by the exchange matrix $B$, which depends on the details of the
Cartan matrix.

To add the coefficients $q_\alpha$, we consider an extended quiver
corresponding to the matrix $\widetilde{B}$, extended by skew-symmetry
and with a vanishing diagonal block connecting the coefficient
nodes. The extended quiver thus has $r$ extra nodes, corresponding to
the coefficients $q_1,...,q_r$. Their connectivity to the other nodes
{\em does not depend} on the details of the Cartan matrix or $B$.

For each $\al$, node $q_\alpha$ is connected only to nodes $\alpha$
and $\overline{\alpha}=\alpha+r$. At the node labeled by $k$ in $\mathcal G_r$,
the connectivity of these three nodes is illustrated by the subquiver
in the left hand side of the figure below. The mutations $\mu_\alpha$
and $\mu_{\overline{\alpha}}$ are the only ones which act on it nontrivially, as
follows:
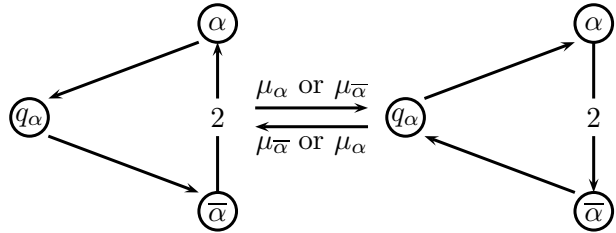
\begin{center}
\psset{unit=2.5mm,linewidth=.4mm,dimen=middle}
\begin{pspicture}(0,-2)(32,12)
\rput(0,5){\pscircle(0,0){1}}
\rput(0,5){$q_\al$}
\rput(10,0){\pscircle(0,0){1}}
\rput(10,0){$\overline{\al}$}
\rput(10,10){\pscircle(0,0){1}}
\rput(10,10){$\al$}
\psline{->}(1,4)(9,1)
\psline{<-}(1,6)(9,9)
\psline(10,1)(10,4)
\psline{->}(10,6)(10,9)
\rput(10,5){$2$}
\psline{->}(12,5.5)(18,5.5)
\rput(15,6.5){$\mu_\al$ or $\mu_{\overline{\alpha}}$}
\psline{<-}(12,4.5)(18,4.5)
\rput(15,3.5){$\mu_{\overline{\al}}$ or $\mu_\alpha$}
\rput(20,5){\pscircle(0,0){1}}
\rput(20,5){$q_\al$}
\rput(30,0){\pscircle(0,0){1}}
\rput(30,0){$\overline{\al}$}
\rput(30,10){\pscircle(0,0){1}}
\rput(30,10){$\al$}
\psline{<-}(21,4)(29,1)
\psline{->}(21,6)(29,9)
\psline{<-}(30,1)(30,4)
\psline(30,6)(30,9)
\rput(30,5){$2$}
\end{pspicture}
\end{center}
After acting with $\mu_\al$ or $\mu_{\overline{\alpha}}$ on the quiver on the
left, all arrows are 
reversed, and {\em vice versa}. 

Now we note that for both cases of simply-laced Lie algebras and the
non simply-laced ones, the evolution from the node $k$ to the node
$k+1$ always involves a sequence where $\mu_\alpha$ acts first, then
$\mu_{\overline{\alpha}}$, in this order, once (in the case of ADE), twice (in
the case of the short roots of B,C,F) or
three times (in the case of the short root of $G_2$). Thus, we return
to the original configuration on the left hand side of the picture
each time, for each of the triples $\alpha,\overline{\alpha},q_\alpha$.

This proves that the rows corresponding to $q_\alpha$ in the mutation
matrix $\widetilde{B}$ are identical at the nodes $k$ and $k+1$, and
hence so is the full exchange matrix.

It is left only to check that, when progressing from smaller to larger
$k$ (i.e. to the ``right'' in the graph $\mathcal G_r$), the
coefficient $q_\alpha$ always appears only in the second term. This is
again due to the order of acting with the even and odd mutations in the graph.

We conclude that the exchange relation with coefficients
\eqref{Qsyscoeff} is preserved under evolutions in the same subgraph
$\mathcal G_r$ which describes $Q$-system evolutions, for any of the
Cartan matrices corresponding to simple Lie algebras.

\section{Generalized $T$-systems, bipartite cluster algebras and the
  polynomial property}

There is a larger class of recursion relations which satisfy the same
polynomiality property, again due to the same argument as in Section
4.  In this section, we formulate these relations, which generalize
the $T$-systems of quantum spin chains \cite{KR,KunibaYsys}, as
cluster algebras. 

For clarity of presentation, we only consider
systems corresponding to simply-laced Lie algebras $\g$ here.  These
are ones which have the simple bipartite graph similar to $\mathcal
G_r$.

The $Q$-systems considered in this paper are the ``combinatorial
limit'' of such $T$-systems, which are the fusion relations satisfied
by the transfer matrices of generalized Heisenberg spin chains. The
combinatorial limit is obtained by dropping one of the parameters,
corresponding in the original system to the spectral parameter.

It is known \cite{Hernandez} that the solutions to the $T$-systems,
given appropriate boundary conditions, are the $q$-characters of
quantum affine algebras \cite{FrenkelResh}. These boundary conditions
are precisely the ones which ensure that the cluster variables have
the polynomiality property.

There is also a similar class of such systems which appears in
representation theory, as shown in \cite{GLS} (section 18.2). As for
the $T$-systems of simply-laced Lie algebras, this example also has a
``bipartite property''. These are mentioned in the second example
below. Interestingly, these systems also are presented in \cite{GLS}
with the precise boundary conditions which ensure polynomiality.

In general, $T$-systems for non-simply laced Lie algebras \cite{KunibaYsys}
are not bipartite, in the same way that $Q$-systems for non
simply-laced algebras are not, and have a more complicated structure
for the graph $\mathcal G_r$.  In order to keep the discussion clear,
we limit ourselves to bipartite systems in this Appendix.

\begin{defn}\label{GeneralizedT} {\bf Generalized $T$-systems.}
We consider the index set $I_r\times \Z$ and a matrix $A$ (the
``incidence matrix'') with rows and columns parametrized by this index
set, $A=(A_{\alpha,\beta}^{i,j})_{\alpha,\beta\in I_r}^{i,j\in\Z}$,
with entries in $\Z_+$. We define a generalized $T$-system with
coefficients to be the recursion relation of the form
\begin{equation}\label{usualT}
T_{\alpha,j;k+1}T_{\alpha,j;k-1} =
T_{\alpha,j+1;k}T_{\alpha,j-1;k}+ q_\alpha
\prod_{j'}\prod_{\beta\neq \al}
(T_{\beta,j';k})^{A_{\beta,\alpha}^{j',j}}, \quad \al\in I_r,\ \ j,k\in\Z.
\end{equation}
\end{defn}

This is a bipartite system, in the following sense.
We view
Equation \eqref{usualT} as an {\rm evolution in the direction of
$k$}. This means that (1) the evolution replaces $T_{\al,j; k-1}$ by
$T_{\al,j; k+1}$, preserving the parity of $k$ and (2) the right hand
side of \eqref{usualT} depends only on variables of the opposite
parity in $k$. The second property is the characteristic of bipartite
systems.

Such a property holds for $Q$-systems with a symmetric Cartan
matrix. We give two examples of $T$-systems with this property.  The
first example appeared in \cite{KR} and was generalized in
\cite{KunibaYsys}. It has generalizations to other Cartan matrices, which we
do not consider.
\begin{example}\label{LieT}
The $T$-systems satisfied by the $q$-characters of Kirillov-Reshetikhin
modules have the following form, in the simply-laced case:
$$
T_{\alpha,j;k+1}T_{\alpha;j,k-1} =
T_{\alpha,j+1;k}T_{\alpha,j-1;k} - \prod_{\beta\sim 
  \alpha} T_{\beta,j;k}^{[-C_{\beta,\alpha}]_+}, \quad \al\in I_r,
k,j\in Z,
$$ where $C$ is the (symmetric) Cartan matrix corresponding to a
simple, simply-laced Lie algebra $\g$. 
\end{example}
\begin{remark}
We have made the identification $T_{\alpha,j;k} = T_{\alpha,k}(u)$
(for rational $R$-matrices)
where $j=2u+c$ for some complex number $C$, where $T_{\alpha,k}(u)$ is
the transfer matrix corresponding to the Kirillov Reshetikhin module
with highest weight $k \omega_\alpha$ and spectral parameter $u$.
Similarly, for trigonometric $R$-matrices, $j$ is related to log of
the spectral parameter with a complex shift.
\end{remark}

The $q$-characters of the KR-modules of the quantum affine algebra of
$\g$, with highest $\g$-weight $k\omega_\alpha$, satisfy this
recursion relation subject to the boundary conditions
$T_{\alpha,j,0}=1$ and $T_{0,j;k}=T_{r+1,j;k}=1$ \cite{Hernandez}.

This example is of the form \eqref{usualT} if we make the identification
$q_\al=-1$ for all $\al\in I_r$ and
$A_{\alpha,\beta}^{ij}=\delta_{ij} [-C_{\alpha,\beta}]_+.$

\begin{example}\label{Glst}
Consider some arbitrary acyclic quiver with nodes labeled by $I_r$ and
 the corresponding quiver matrix $\Gamma=(\Gamma_{ij})_{i,j\in I_r}$, such that
 $\Gamma_{i,j}=k>0$ if there are $k$ arrows pointing from node $i$ to
 node $j$, and $\Gamma$ is a skew-symmetric matrix.  

To such a quiver
 there corresponds a recursion relation of the form \cite{GLS}
\begin{equation}\label{fsys}
f_{i,a-1,b}f_{i,a,b-1}=f_{i,a,b}f_{i,a-1,b-1}-\prod_{j:i\rightarrow j}
f_{j,a,b} \prod_{j:j\rightarrow i} f_{j,a-1,b-1}.
\end{equation}
Here, the product over $i:i\rightarrow j$ has $k$ factors if
$\Gamma_{ij}=k>0$, and so forth.

Defining $f_{i,a,b}=:T_{i,a+b+1;b-a}$, this recursion relation can be
written in the form of a $T$-system
$$
T_{\alpha,j;k+1} T_{\alpha,j;k-1} = T_{\alpha,j+1;k} T_{\alpha,j-1;k} -
\prod_{\beta}\left( T_{\beta,j+1;k}^{[\Gamma_{\alpha,\beta}]_+}
 T_{\beta,j-1;k}^{[\Gamma_{\beta,\alpha}]_+}\right).
$$ 
Define the matrix $A$ as follows:
$A_{\beta,\alpha}^{j+1,j}=[\Gamma_{\alpha,\beta}]_+$,
$A_{\beta,\alpha}^{j-1,j}=[\Gamma_{\beta,\alpha}]_+$, and all other
entries of $A$ vanishing. Then the last equation takes the form
\eqref{usualT}:
$$
T_{\alpha,j;k+1} T_{\alpha,j;k-1} = T_{\alpha,j+1;k} T_{\alpha,j-1;k}
-
\prod_{j'}\prod_{\beta}(T_{\beta,j';k})^{A_{\beta,\alpha}^{j',j}}.
$$
\end{example}
This last relation is obviously equation \eqref{usualT} specialized so
that all the coefficients $q_\al=-1$. We remark that the way in which
the identification of the indices was made above, there is a parity
restriction on the sum $j+k$. This is not essential to the discussion below,
as the two parities of $j+k$ actually decouple in the associated cluster algebra.

\subsection{Formuation of generalized $T$-systems as cluster
  algebras}
The bipartite $T$-systems of Equation \eqref{usualT} 
can be formulated as cluster algebras of infinite rank, if the matrix
$A$ satisfies some mild conditions (see the Lemma below). The model
for this description is the $Q$-system in the previous section for the
simply-laced case.

First, we specify a cluster variable (including coefficients) at some node
in the cluster graph,
which we label $k$, and an augmented exchange matrix $\widetilde{B}$,
which includes the submatrix $B$.
Equation \eqref{usualT} is an exchange relation
$\mu_{\alpha,j}: T_{\alpha,j;k-1}\mapsto T_{\alpha,j;k+1}$, and
the parity of $k$ is preserved under any mutation. Thus, the exchange
matrix $\widetilde{B}$ consists of two column-sets, each labeled by
$I_r\times \Z$, corresponding to even and odd variables ($k$ is fixed):
\begin{eqnarray*}
\mathbf T_{\rm even} &=& \{ T_{\alpha,j;2k} ~:~ \alpha\in I_r, j\in\Z\},\\
\mathbf T_{\rm odd} &=&\{ T_{\alpha,j;2k+1} ~:~ \alpha\in I_r, j\in\Z\}.
\end{eqnarray*}
We also have $r$ coefficients $\{q_\alpha: \alpha\in I_r\}$ which do
not mutate. At the node labeled $k$ in the cluster tree, the cluster
variable is 
\begin{equation}\label{clustervar}
\bx=(\mathbf T_{\rm even}, \mathbf T_{\rm
  odd};q_1,...,q_r).
\end{equation}

Let $A$ be a matrix as in Definition \ref{GeneralizedT} and let $P$ be
the matrix on the same index set $I_r\times\Z$ defined as
\begin{equation}\label{Pmatrix}
P_{\alpha,\beta}^{i,j} = \delta_{\alpha,\beta}(\delta_{i,j+1}+\delta_{i,j-1}).
\end{equation}
Define $C=P-A$, then the exchange matrix $B$ on the ``doubled'' index
set is 
\begin{equation}\label{Bmat}
B=\begin{pmatrix}0 & -C^t \\ C & 0 \end{pmatrix}
\end{equation}
at the node $k$.

In order to distinguish between the index sets for the ``even'' and
``odd'' variables, we will refer to the first as $(\alpha,j)$ and the
latter as $(\overline{\alpha},j)$ in analogy with the formulation of the
$Q$-system. Even mutations act according to the first set of columns,
and odd mutations act via the second half.

The augmented matrix $\widetilde{B}$ is defined by adding $r$ rows to
$B$, which we label simply by the indices $q_1,...,q_r$. For example, the entry
$B_{q_\alpha, (\beta,j)}$ corresponds to the row labeled by
$q_\alpha$ and column $(\beta,j)$ of an even variable, whereas 
$B_{q_\alpha, (\overline{\beta},j)}$ is in the column
corresponding to the odd variable $(\overline{\beta},j)$.
Thus, at the
node $k$,
\begin{equation}\label{Bext}
\widetilde{B}_{q_\alpha,(\beta,j)}=
-\delta_{\alpha,\beta}=-\widetilde{B}_{q_\alpha,(\overline{\beta},j)}. 
\end{equation}

We can write the $T$-system \eqref{usualT} as
$$ T_{\alpha,j;k+1} T_{\alpha,j;k-1} =
\prod_{j'}T_{\alpha,j';k}^{[C_{\alpha,\alpha}^{j',j}]_+}+ q_\alpha
\prod_{\alpha,j'} T_{\beta,j';k}^{[-C_{\beta,\alpha}^{j',j}]_+}.
$$ This is identical to the exchange relation
$\mu_{{\alpha},j}:T_{\alpha,j;k-1}\mapsto T_{\alpha;j,k+1}$
given by the 
matrix $\widetilde{B}$ above when $k$ is odd, and
$\mu_{\overline{\alpha},j}:T_{\alpha,j;k+1}\mapsto T_{\alpha,j;k-1}$ if $k$ is
even.

The mutations
$\mu_{\alpha,j}$ commute with each other for different
$(\alpha,j)$. Define the node $k'$ to be the node reached from $k$ via
the compound mutation
$$\mu_{\rm even}=\prod_{\al\in I_r,j\in\Z}\mu_{\al,j}.$$ 
The cluster variable at this
node is $\bx' = (\mathbf T_{\rm even}', \mathbf T_{\rm odd},
q_1,...,q_r)$, where $\mathbf T_{\rm even}'=\{T_{\al,j,2k+2}: \al\in
I_r,j\in \Z\}$, and the
rest of the entries are as in $\bx$ at node $k$.

\begin{defn}
The ``bipartite'' subgraph $\mathcal G_r$ of the full tree associated
with the cluster algebra is the following: It is the subgraph
containing the node labeled $k\in \Z$, as well as all nodes obtained
from it by mutations along all distinct even edges, or alternatively
all distinct odd edges.

The node $k'$ in $\mathcal G_r$ (respectively $(k-1)'$) is the node
reached from node $k$ by the sequence of all even (respectively odd)
mutations.  The node $k+1$ in $\mathcal G_r$ is the node reached from
$k'$ by the sequence of all odd mutations. The graph is thus extended
for all $k\in\Z$.
\end{defn}

The cluster algebra structure holds iff we require that the exchange
relation given by the matrix $\widetilde{B}'$ at the node $k'$ be
consistent with the evolution $\mu_{\overline{\al},j}:
T_{{\al},j;2k+1}\mapsto T_{{\al},j;2k+3}$ given by equation
\eqref{usualT}.  The mutations are consistent with the recursion
\eqref{usualT} if and only if
$\widetilde{B}'=-\widetilde{B}$. This condition is satisfied if the
three conditions of the following lemma hold.

\begin{lemma}
The mutations along $\mathcal G_r$ of the cluster algebra defined by
 the seed $(\bx,\widetilde{B})$ at node $k$, with $\bx$ as in Equation
 \eqref{clustervar} and $\widetilde{B}$ is as above,
 restricted to the edges of the subgraph $\mathcal G_r$, are each
 described by one of the recursion relations \eqref{usualT}, with
the matrix $C=P-A$, where $P$ is given by Equation
 \eqref{Pmatrix} and $A$ is a matrix with non-negative entries,
 provided that:
\begin{enumerate}
\item The matrices $A$ and $P$ satisfy the condition $P A^T - A
  P=0$. (Given condition 3, this implies that $A$ and $P$ commute.)
\item The matrix $P$ is such that $\sum_k
  P_{\alpha\beta}^{kj}=2\delta_{\alpha,\beta}$. 
\item The matrix $A$ is symmetric.
\end{enumerate}
\end{lemma}
\begin{proof}
These three conditions are necessary and sufficient to ensure that 
$\widetilde{B}'=-\widetilde{B}$. 

Under the compound mutation $\mu_{\rm even}$, the even rows and
columns change sign. 
We are left with two other blocks of $\widetilde{B}'$ whose mutation needs to
be checked: The odd columns-odd rows, and the odd columns and the
coefficient rows.

Condition (1) comes from requiring The first condition comes from
requiring that $(B')_{\overline{\alpha},\overline{\beta}}^{jl} =
0$. To see this, consider
\begin{eqnarray*}
(B')_{\overline{\alpha},\overline{\beta}}^{jl} &=& 0 +
    \sum_{\gamma,k}{\rm sign}(B_{\overline{\alpha},\gamma}^{jk} )
    [B_{\overline{\alpha},\gamma}^{jk} B_{\gamma,\overline{\beta}}^{kl}]_+ \\
    &=& \sum_{\gamma,k}{\rm \sign}(
    C_{\alpha,\gamma})[-C_{\alpha,\gamma}^{jk} C_{\beta,\gamma}^{lk}]_+
    \\ &=& (P A^t - A P)_{\alpha\beta}^{jl}
\end{eqnarray*}

Condition (2) is the result of requiring that
$(\widetilde{B}')_{q_\alpha,(\overline{\beta},i)}=
-\widetilde{B}_{q_\alpha,(\overline{\beta},i)}$. 
It is easiest to illustrate this graphically (see below), but also,
\begin{eqnarray*}
(\widetilde{B}')_{q_\alpha,(\overline{\beta},i)} &=& \delta_{\alpha,\beta} +
\sum_{\gamma,k} {\rm
  sign}(B_{q_\alpha,(\gamma,k)})[B_{q_\alpha,(\gamma,k)}
  B_{\gamma,\overline{\beta}}^{ki}]_+ \\
&=&  \delta_{\alpha,\beta} -\sum_{\gamma,k}[\delta_{\alpha,\gamma}
  C_{\beta,\alpha}^{ik}]_+ \\
&=& \delta_{\alpha,\beta} - \sum_k P_{\beta,\alpha}^{ik}.
\end{eqnarray*}
Requiring that
$(\widetilde{B}')_{q_\al,(\overline{\beta},i)}=-\delta_{\al,\beta}$
gives condition (2) of the Lemma.

Condition (3) appears because we require the evolution
$\mu_{\overline{\alpha},j}$ to have the same form at $k'$ as the
evolution $\mu_{\alpha,j}$ at the node $k$. That is, the evolution of
even and odd variables is the same, as the matrix $A$ in equation
\eqref{usualT} does not depend on $k$. Therefore, since the upper
right block of $\widetilde{B}'$ is equal to $C^t=P-A^t$, we find that
$A=A^t$. 
\end{proof}

We illustrate the last statement of the proof graphically in terms of
the quiver graph corresponding to $\widetilde{B}$. In fact, we only
need to consider the cluster variables in $\bx$ which are connected to
$q_\alpha$. These are $\{T_{\alpha,j;2k},T_{\alpha,j;2k+1}: j\in
\Z\}$, for some fixed $k$ and $\alpha$.
(In the system of \cite{GLS} there is a
restriction on the parity of $j+k$ but this is not a necessary
assumption for our discussion here, as the two parities decouple). 

Thus, we need only consider a ``slice'' of the quiver graph with
constant $\alpha$. The quiver graph of this slice is an infinite
double strip and one extra node, has the following arrows at
the node $k$ (to keep the picture readable, we omit the arrows between
the empty circles and 
$q_\alpha$ as they are clearly decoupled from what happens to the
quiver with solid circles. They are connected in exactly the same way
as the solid circles):
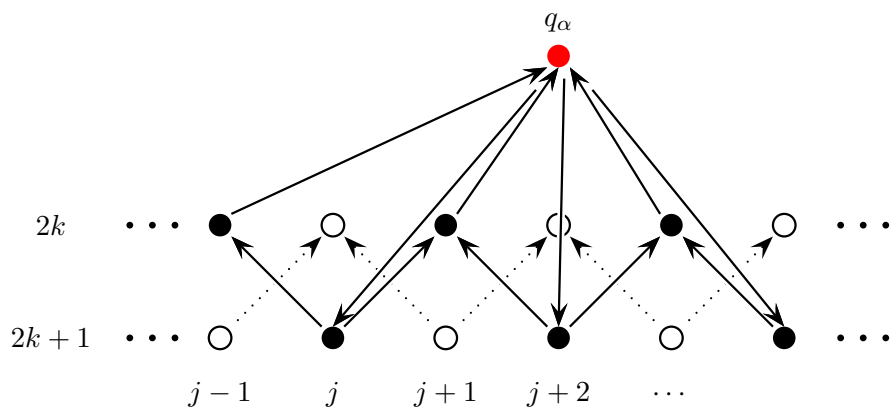
\begin{center}
\psset{unit=1.5mm,linewidth=.3mm,dimen=middle}
\begin{pspicture}(-20,-10)(70,30)
\multips(10,0)(20,0){3}{\rput(0,0){\pscircle*(0,0){1}}}
\multips(0,0)(20,0){3}{\rput(0,0){\pscircle(0,0){1}}}
\multips(0,10)(20,0){3}{\rput(0,0){\pscircle*(0,0){1}}}
\multips(10,10)(20,0){3}{\rput(0,0){\pscircle[border=1pt](0,0){1}}}
\rput(10,-5){$j$}
\rput(0,-5){$j-1$}
\rput(20,-5){$j+1$}
\rput(30,-5){$j+2$}
\rput(40,-5){$\cdots$}
\multips(55,0)(2,0){3}{\rput(0,0){\pscircle*(0,0){0.3}}}
\multips(55,10)(2,0){3}{\rput(0,0){\pscircle*(0,0){0.3}}}%
\multips(-8,0)(2,0){3}{\rput(0,0){\pscircle*(0,0){0.3}}}
\multips(-8,10)(2,0){3}{\rput(0,0){\pscircle*(0,0){0.3}}}
\rput(-15,0){$2k+1$}
\rput(-15,10){$2k$}
\multips(0,0)(20,0){3}{\psline[linestyle=dotted,arrowsize=3pt
    4]{->}(1,1)(9,9)}
\multips(0,0)(20,0){2}{\psline[linestyle=dotted,arrowsize=3pt 4]{->}(19,1)(11,9)}
\multips(0,0)(20,0){3}{\psline[border=1pt,arrowsize=3pt 4]{->}(9,1)(1,9)}
\multips(0,0)(20,0){2}{\psline[border=1pt,arrowsize=3pt 4]{->}(11,1)(19,9)}
\rput(30,25){\pscircle*[linecolor=red](0,0){1}}
\psline[border=1pt,arrowsize=3pt 4]{<-}(10,1)(28,22)
\psline[border=1pt,arrowsize=3pt 4]{<-}(30,1)(30.5,23)
\psline[border=1pt,arrowsize=3pt 4]{<-}(50,1)(33,22)
\psline[border=1pt,arrowsize=3pt 4]{->}(1,11)(29,24)
\psline[border=1pt,arrowsize=3pt 4]{->}(21,11)(30,24)
\psline[border=1pt,arrowsize=3pt 4]{->}(39,11)(31,24)
\rput(30,28){$q_\alpha$}
\end{pspicture}\end{center}
Note that the sum of the vertical and horizontal coordinates has a
different parity for solid and empty circles. In the slice of constant
$\alpha$ they are not connected by any arrows.

Now suppose we act with $\prod_{j} \mu_{\alpha,j}$ on this
quiver. Then all the arrows reverse. This is seen from the following
sequence of mutations on a piece of the quiver:
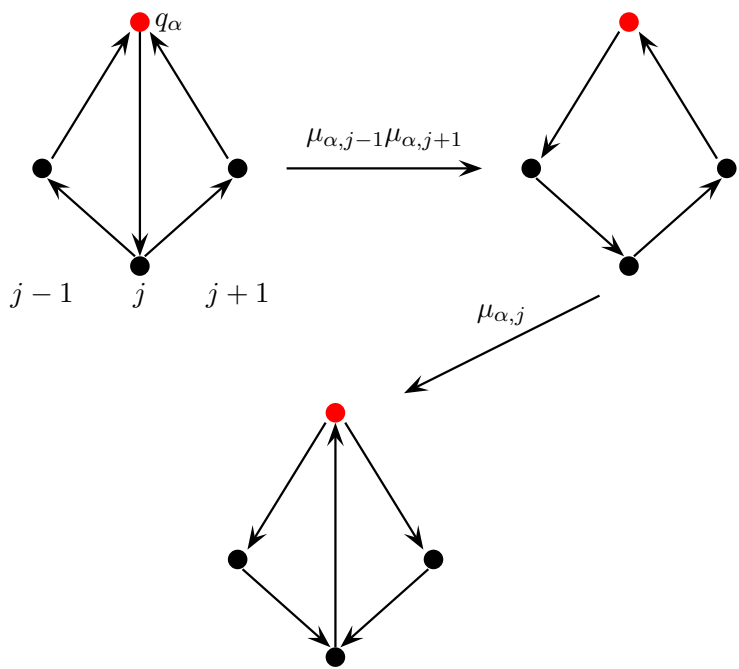
\begin{center}
\psset{unit=1.3mm,linewidth=.3mm,dimen=middle}
\begin{pspicture}(0,-5)(70,80)
\rput(0,50){\rput(10,0){\pscircle*(0,0){1}}
\rput(10,-3){$j$}
\rput(0,10){\pscircle*(0,0){1}}
\rput(0,-3){$j-1$}
\rput(20,-3){$j+1$}
\rput(20,10){\pscircle*(0,0){1}}
\rput(10,25){\pscircle*[linecolor=red](0,0){1}}
\rput(13,25){$q_\alpha$}
\psline[arrowsize=3pt 4]{->}(9.5,1)(0.5,9)
\psline[arrowsize=3pt 4]{->}(10.5,1)(19.5,9)
\psline[arrowsize=3pt 4]{->}(1,11)(9,24)
\psline[arrowsize=3pt 4]{->}(19,11)(11,24)
\psline[arrowsize=3pt 4]{->}(10,24)(10,1)
\psline[arrowsize=3pt 4]{->}(25,10)(45,10)}
\rput(35,63){$\mu_{\alpha,j-1}\mu_{\alpha,j+1}$}
\rput(50,50){\rput(10,0){\pscircle*(0,0){1}}
\rput(0,10){\pscircle*(0,0){1}}
\rput(20,10){\pscircle*(0,0){1}}
\rput(10,25){\pscircle*[linecolor=red](0,0){1}}
\psline[arrowsize=3pt 4]{<-}(9.5,1)(0.5,9)
\psline[arrowsize=3pt 4]{->}(10.5,1)(19.5,9)
\psline[arrowsize=3pt 4]{<-}(1,11)(9,24)
\psline[arrowsize=3pt 4]{->}(19,11)(11,24)}
\psline[arrowsize=3pt 4]{->}(57,47)(37,37)
\rput(47,45){$\mu_{\alpha,j}$}
\rput(20,10){\rput(10,0){\pscircle*(0,0){1}}
\rput(0,10){\pscircle*(0,0){1}}
\rput(20,10){\pscircle*(0,0){1}}
\rput(10,25){\pscircle*[linecolor=red](0,0){1}}
\psline[arrowsize=3pt 4]{<-}(9.5,1)(0.5,9)
\psline[arrowsize=3pt 4]{<-}(10.5,1)(19.5,9)
\psline[arrowsize=3pt 4]{<-}(1,11)(9,24)
\psline[arrowsize=3pt 4]{<-}(19,11)(11,24)
\psline[arrowsize=3pt 4]{<-}(10,24)(10,1)}
\end{pspicture}\end{center}
All the arrows are reversed in the last picture. This shows that these
entries of the matrix $\widetilde{B}'$ indeed have the opposite sign.

We conclude that generalized $T$-systems of the form \eqref{usualT} are
subsets of cluster algebras, so we can apply the Lemmas of Section 4:
\begin{lemma}
In the case where there is a boundary condition on the recursion relation
\eqref{usualT} such that $T_{\alpha, j;k_0-1}=0$ for all $j$ and
$\alpha$, with some fixed $k_0$, then all cluster 
variables are polynomials in the variables $\{T_{\alpha,j;k_0}:
\alpha\in I_r, j\in \Z\}$.
\end{lemma}

It turns out that in both the examples given in this appendix, such a
boundary condition holds with $k_0=0$. This gives a simple explanation
of the phenomenon of polynomiality of the solutions of the recursion
found in \cite{GLS}. 

\end{appendix}

\def\cprime{$'$} \def\cprime{$'$} \def\cprime{$'$} \def\cprime{$'$}
  \def\cprime{$'$} \def\cprime{$'$} \def\cprime{$'$} \def\cprime{$'$}
  \def\cprime{$'$} \def\cprime{$'$}
\providecommand{\bysame}{\leavevmode\hbox to3em{\hrulefill}\thinspace}
\providecommand{\MR}{\relax\ifhmode\unskip\space\fi MR }
\providecommand{\MRhref}[2]{%
  \href{http://www.ams.org/mathscinet-getitem?mr=#1}{#2}
}
\providecommand{\href}[2]{#2}


\end{document}